\documentclass[10pt]{amsart}
\usepackage{amssymb}
\usepackage{amsmath,amssymb,amsfonts,amsthm,
latexsym, amscd, amsfonts, epsfig, epic}
\usepackage{mathrsfs}
\usepackage{color}
\usepackage{eucal}
\usepackage{eufrak}
\usepackage{xspace}
\usepackage{a4wide}
\usepackage{colordvi}
\usepackage{eucal}

\newcounter{mycount}
\theoremstyle{plain}
\newtheorem{theorem}[mycount]{Theorem}

\newtheorem{lemma}[mycount]{Lemma}

\theoremstyle{definition}

\theoremstyle{example}

\theoremstyle{remark}

\numberwithin{equation}{section} \numberwithin{figure}{section}

\begin{document}
\title{A bijective enumeration of $3$-strip tableaux}
\author{Emma Yu Jin$^{\dag}$}
\thanks{$^{\dag}$ Corresponding author email: yu.jin@tuwien.ac.at. Tel.: $+43(1)58801-104583$.
The author was supported by the German Research Foundation DFG, JI 207/1-1, and the Austrian Research
Fund FWF, project SFB F50 Algorithmic and Enumerative Combinatorics}
\email{yu.jin@tuwien.ac.at}
\address{Institut f\"ur Diskrete Mathematik und Geometrie, TU Wien, Wiedner Hauptstr. 8--10, 1040 Wien, Austria}

\maketitle

\begin{abstract}
Baryshnikov and Romik derived the combinatorial identities for the numbers of the $m$-strip tableaux. This generalized the classical Andr\'{e}'s theorem for the number of up-down permutations.
They asked for a bijective proof for the enumeration of $3$-strip tableaux. In this paper we will provide such a bijective proof. First we count the $3$-strip tableaux by decomposition. Secondly we will apply this ``decomposition'' idea on the up-down permutations and down-up permutations to enumerate the $3$-strip tableaux bijectively.
\end{abstract}

\section{Introduction}
A permutation $w=a_1a_2\cdots a_n\in\mathfrak{S}_n$ is called an {\it
up-down permutation} if $a_1<a_2>a_3<a_4>\cdots$. It is well-known
that the number of up-down permutations of $[n]$ is Euler number
$E_n$, whose exponential generating function is
\begin{eqnarray}\label{E:andre}
\sum_{n\ge 0}E_n\frac{x^n}{n!}=\sec x+\tan x.
\end{eqnarray}
This is also called Andr\'{e}'s theorem \cite{And:79}. Sometimes $E_{2n}$ is called a secant number
and $E_{2n+1}$ a tangent number. A permutation $w=a_1a_2\cdots a_n\in\mathfrak{S}_n$ is called a {\it down-up permutation} if $a_1>a_2<a_3>a_4<\cdots$. The bijection
$a_1a_2\cdots a_{n}\mapsto (n+1-a_1)\,(n+1-a_2)\,\cdots (n+1-a_n)$ transforms a down-up permutation into an up-down permutation. Up-down permutations can be thought of as a special case of a standard
Young tableau. For instance, the permutation
$\sigma=132546$ is an up-down
permutation, which can be identified as the tableau below.
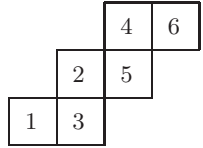
\begin{center}
\setlength{\unitlength}{3pt}
\begin{picture}(-30,20)(10,25)
\put(-22,25){\line(1,0){12}}\put(-22,25){\line(0,1){6}}
\put(-16,25){\line(0,1){12}}\put(-10,25){\line(0,1){12}}
\put(-22,31){\line(1,0){18}}\put(-4,31){\line(0,1){12}}
\put(-16,37){\line(1,0){12}}\put(-10,37){\line(0,1){6}}
\put(-10,43){\line(1,0){12}}\put(2,43){\line(0,-1){6}}
\put(2,37){\line(-1,0){6}}
\put(-20,27){\small{$1$}}\put(-14,27){\small{$3$}}
\put(-14,33){\small{$2$}}\put(-8,33){\small{$5$}}
\put(-8,39){\small{$4$}}\put(-2,39){\small{$6$}}
\end{picture}
\end{center}
We adopt the notations from \cite{BR:07}. Formally speaking, an {\it integer partition} is a sequence
$\lambda=(\lambda_1,\lambda_2,\ldots,\lambda_k)$ where $\lambda_1\ge
\lambda_2\ge \cdots\ge\lambda_k>0$ are integers. We identify each
partition $\lambda$ with its Young diagram and speak of them interchangeably. Given a partition $\lambda=(\lambda_1,\lambda_2,\ldots,\lambda_k)$, the {\it Young diagram} of shape $\lambda$ is a left-justified array of $\lambda_1+\lambda_2+\cdots+\lambda_n$ boxes with $\lambda_1$ in the first row, $\lambda_2$ in the second row, and so on. A {\it skew Young diagram} is the difference $\lambda/\mu$
of two Young diagrams where $\mu\subset\lambda$. If
$\lambda/\mu$ is a skew Young diagram, a {\it standard Young tableau} of shape $\lambda/ \mu$ is a filling of the boxes of $\lambda / \mu$ with the integers
$1,2,\ldots,\vert\lambda / \mu\vert$ that is increasing
along rows and columns, where $\vert\lambda / \mu\vert$ is the number of boxes of shape $\lambda / \mu$ and is called the size of shape $\lambda / \mu$. Given any skew shape $\lambda/\mu$ of size $n$, let $f^{\lambda/\mu}$ denote the number of standard Young tableaux of shape $\lambda/\mu$, i.e., the number of ways to put $1,2,\ldots,n$ into the squares of the diagram of $\lambda/\mu$, each number $1,2,\ldots,n$ occurring exactly once, so that the rows and columns are increasing. Given a standard Young tableau, we can form the {\it reading word} of the tableau by reading the bottom row from left to right, then the next-to-bottom row, and so on. The reading word of the above tableau is exactly the permutation $\sigma=132546$.

Up-down permutations of $[n]$ are in simple bijection with standard Young tableaux of shape
$$\theta_n=(m+1,m,m-1,\ldots,3,2)/ (m-1,m-2,\ldots,1,0)$$ when $n=2m$ is even, or $$\theta_n=(m,m,m-1,m-2,\ldots,3,2)/
(m-1,m-2,\ldots,1,0)$$
when $n=2m-1\ge 5$ is odd and $\theta_1=(1)$, $\theta_3=(2,2)/(1,0)$. Clearly this bijection converts each standard Young tableau of shape $\theta_n$ into an up-down permutation via its reading word. By ``thickening'' the shape $\theta_n$, Baryshnikov and Romik generalized the classical enumeration formula (\ref{E:andre}) for up-down permutations \cite{BR:07}. They introduced the $m$-strip tableaux and enumerated the $m$-strip tableaux by using transfer operators, but the computations become more complicated as $m$ increases. The standard Young tableaux of shape $\theta_n$ are exactly $2$-strip tableaux, which are counted by the Euler number $E_n$. For the particular $3$-strip tableaux, there are three different shapes of $3$-strip tableaux, denoted by $\sigma_{3n-2},\sigma_{3n-1},\sigma_{3n}$, respectively. Let $\sigma_{3n-2}$ be the Young diagram of shape
$$(m,m,m-1,m-2,\ldots,3,2)/(m-2,m-3,\ldots,1,0,0)$$
that contains $3n-2$ boxes when $m\ge 3$, and $\sigma_1$ be the Young diagram of shape $(1)$, $\sigma_4$ be the Young diagram of shape $(2,2)$. Let $\sigma_{3n-1}$ be the Young diagram of shape $$(m,m,m,m-1,m-2,\ldots,3,2)/(m-1,m-2,\ldots,1,0,0)$$
that contains $3n-1$ boxes when $m\ge 3$, and $\sigma_2$ be the Young diagram of shape $(1,1)$, $\sigma_5$ be the Young diagram of shape $(2,2,2)/(1,0,0)$. Let furthermore
$\sigma_{3n}$ be the Young diagram of shape
$$(m,m,m,m-1,\ldots,3,2,1)/(m-1,m-2,\ldots,1,0,0,0)$$
that contains $3n$ boxes when $m\ge 3$, and $\sigma_3$ be the Young diagram of shape $(1,1,1)$, $\sigma_6$ be the Young diagram of shape $(2,2,2,1)/(1,0,0,0)$. Below we show three standard Young tableaux of shape $\sigma_{7},\sigma_{8},\sigma_9$, from left to right, respectively.
\begin{center}
\setlength{\unitlength}{3pt}
\begin{picture}(5,30)(10,20)
\put(-32,24){\line(1,0){12}}\put(-32,24){\line(0,1){12}}
\put(-26,24){\line(0,1){18}}\put(-20,24){\line(0,1){18}}
\put(-32,30){\line(1,0){18}}\put(-32,36){\line(1,0){18}}
\put(-26,42){\line(1,0){12}}\put(-14,42){\line(0,-1){12}}
\put(-29.5,26){\small{$2$}}\put(-29.5,32){\small{$1$}}
\put(-23.5,32){\small{$4$}}\put(-23.5,26){\small{$5$}}
\put(-23.5,38){\small{$3$}}\put(-17.5,38){\small{$6$}}
\put(-17.5,32){\small{$7$}}\put(-23.5,46){\small{$\sigma_7$}}
\put(-2,21){\line(1,0){12}}\put(-2,21){\line(0,1){12}}
\put(4,21){\line(0,1){18}}\put(10,21){\line(0,1){24}}
\put(-2,27){\line(1,0){18}}\put(-2,33){\line(1,0){18}}
\put(4,39){\line(1,0){12}}\put(16,39){\line(0,-1){12}}
\put(10,45){\line(1,0){6}}\put(16,45){\line(0,-1){6}}
\put(0.5,23){\small{$2$}}\put(0.5,29){\small{$1$}}
\put(6.5,29){\small{$4$}}\put(6.5,23){\small{$5$}}
\put(6.5,35){\small{$3$}}\put(12.5,35){\small{$7$}}
\put(12.5,41){\small{$6$}}\put(12.5,29){\small{$8$}}
\put(0,46){\small{$\sigma_8$}}
\put(28,25){\line(1,0){12}}\put(28,25){\line(0,1){12}}
\put(34,25){\line(0,1){18}}\put(40,25){\line(0,1){24}}
\put(28,31){\line(1,0){18}}\put(28,37){\line(1,0){18}}
\put(34,43){\line(1,0){12}}\put(46,43){\line(0,-1){12}}
\put(40,49){\line(1,0){6}}\put(46,49){\line(0,-1){6}}
\put(28,25){\line(0,-1){6}}\put(28,19){\line(1,0){6}}
\put(34,19){\line(0,1){6}}
\put(30.5,27){\small{$2$}}\put(30.5,33){\small{$1$}}
\put(30.5,21){\small{$3$}}\put(36.5,33){\small{$5$}}
\put(36.5,39){\small{$4$}}\put(36.5,27){\small{$6$}}
\put(42.5,39){\small{$8$}}\put(42.5,45){\small{$7$}}
\put(42.5,33){\small{$9$}}\put(32,46){\small{$\sigma_9$}}
\end{picture}
\end{center}
Baryshnikov and Romik proved
\begin{theorem}[\cite{BR:07}, $3$-strip tableaux]\label{T:3strip}
\begin{align}
\label{E:3s2}f^{\sigma_{3n-2}}&=\frac{(3n-2)!E_{2n-1}}{(2n-1)!2^{2n-2}},\\
\label{E:3s1}f^{\sigma_{3n-1}}&=\frac{(3n-1)!E_{2n-1}}{(2n-1)!2^{2n-1}},\\
\label{E:3s}f^{\sigma_{3n}}&=\frac{(3n)!(2^{2n-1}-1)E_{2n-1}}{(2n-1)!2^{2n-1}
(2^{2n}-1)}.
\end{align}
\end{theorem}
Baryshnikov and Romik \cite{BR:07} asked for a bijective proof of Theorem~\ref{T:3strip}, namely to prove
Theorem~\ref{T:3strip} by directly relating $3$-strip tableaux to up-down permutations in some combinatorial way. Here we first prove Theorem~\ref{T:3strip} by decomposing $3$-strip tableaux. The key decomposition offers an inductive way to relate $3$-strip tableaux and up-down permutations. Furthermore, by applying this ``decomposition'' idea on the up-down permutations and down-up permutations, we provide a purely bijective proof of Theorem~\ref{T:3strip}.

After Baryshnikov and Romik published their results \cite{BR:07}, Stanley \cite{Sta:11} also generalized the Young diagram of shape $\theta_n$ into the skew partition $\sigma(a,b,c,n)$. $\sigma(a,b,c,n)$ is defined to be the skew partition whose Young diagram has $a$ squares in the first row, $b$ squares in the other nonempty rows, and $n$ rows in total. Moreover, each row begins $c-1$ columns to the left of the row above, with $b\ge c$. Although the standard Young tableaux of shape $\sigma(a,b,c,n)$ is neither a subset nor a superset of the $m$-strip tableaux, the skew partition $\sigma(2,3,2,n)$ (resp. $\sigma(3,3,2,n)$) after transposing rows to columns, is exactly $\sigma_{3n-1}$ (resp. $\sigma_{3n}$). It follows that $f^{\sigma(2,3,2,n)}=f^{\sigma_{3n-1}}$ and $f^{\sigma(3,3,2,n)}=f^{\sigma_{3n}}$.

Stanley \cite{Sta:11} derived the generating functions for the standard Young tableaux of shape $\sigma(a,b,c,n)$ by analyzing the determinant in the Aitken formula. For any integer partition $\lambda$, let $\ell(\lambda)$ be the length of $\lambda$. If $\ell(\lambda)\le m$ and $\mu\subseteq \lambda$, the Aitken formula asserts that
$$f^{\lambda/\mu}=n!\det\left[\frac{1}{(\lambda_i-\mu_j-i+j)!}\right]_{i,j=1}^m.$$
The Aikten formula can be obtained by applying the exponential
specialization on the Jacobi-Trudi identity, see \cite{Stanley:ec2,Sta:11}. For the special cases $f^{\sigma(2,3,2,n)}=f^{\sigma_{3n-1}}$ and $f^{\sigma(3,3,2,n)}=f^{\sigma_{3n}}$, the exponential generating functions for $3$-strip tableaux of shape $\sigma_{3n-1}$ and $\sigma_{3n}$ are \cite{Sta:11}
\begin{align}
\label{E:gen2}\sum_{n\ge 1}\frac{f^{\sigma_{3n-1}}}{(3n-1)!}x^{2n}
&=\left[\sum_{n\ge 1}\frac{(-1)^{n-1}x^{2n}}{(2n)!}\right]
\left[\sum_{n\ge 0}\frac{(-1)^{n}x^{2n}}{(2n+1)!}\right]^{-1}=x\tan(\frac{x}{2}),\\
\label{E:gen3}\sum_{n\ge 1}\frac{f^{\sigma_{3n}}}{(3n)!}x^{2n}
&=\left[\sum_{n\ge 1}\frac{(-1)^{n-1}x^{2n}}{(2n+1)!}\right]
\left[\sum_{n\ge
0}\frac{(-1)^{n}x^{2n}}{(2n+1)!}\right]^{-1}=\frac{x}{\sin(x)}-1.
\end{align}
In Section~\ref{S:decom} we will prove
the generating functions for the $3$-strip tableaux of shape $\sigma_{3n-i}$ for $i=0,1,2$
by decomposing the $3$-strip tableaux. In Section~\ref{S:bij} we give a bijective proof of
Theorem~\ref{T:3strip}. Section~\ref{S:bij} is independent of the generating functions in
subsection~\ref{S:sub1}.
\section{The key decomposition}\label{S:decom}
In what follows, we represent each standard Young tableau by a natural labeling on its corresponding poset and posets are depicted as Hasse diagrams. For $i=0,1,2$, let $P_{\sigma_{3n-i}}$ be the poset whose elements are the squares of the Young diagram of shape $\sigma_{3n-i}$, with $t$ covering $s$ if $t$ lies directly to the right or directly below $s$, with no squares in between. In this way, each Young diagram of shape $\sigma_{3n-i}$ can be represented by the Hasse diagram of $P_{\sigma_{3n-i}}$ and we will speak of them interchangeably. A {\it natural labeling} of $P_{\sigma_{3n-i}}$ is an order-preserving bijection $\eta:P_{\sigma_{3n-i}}\rightarrow [3n-i]$, i.e., a natural labeling $\eta$ is a bijection such that $\eta(x)\le \eta(y)$ for every $x,y\in P_{\sigma_{3n-i}}$ and $x\le y$. Sometimes a natural labeling of $P_{\sigma_{3n-i}}$ is also called a {\it linear extension} of $P_{\sigma_{3n-i}}$. The number of natural labelings of $P_{\sigma_{3n-i}}$ is denoted $e(P_{\sigma_{3n-i}})$. Then we have $f^{\sigma_{3n-i}}=e(P_{\sigma_{3n-i}})$ because each $3$-strip tableau of shape $\sigma_{3n-i}$ can be identified as a natural labeling of the poset $P_{\sigma_{3n-i}}$ for $i=0,1,2$, see \cite{Stanley:ec2}. For example, every $3$-strip tableau of
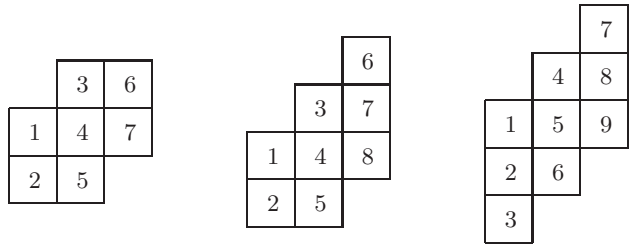
\begin{center}
\setlength{\unitlength}{3pt}
\begin{picture}(5,28)(10,20)
\put(-32,24){\line(1,0){12}}\put(-32,24){\line(0,1){12}}
\put(-26,24){\line(0,1){18}}\put(-20,24){\line(0,1){18}}
\put(-32,30){\line(1,0){18}}\put(-32,36){\line(1,0){18}}
\put(-26,42){\line(1,0){12}}\put(-14,42){\line(0,-1){12}}
\put(-29.5,26){\small{$2$}}\put(-29.5,32){\small{$1$}}
\put(-23.5,32){\small{$4$}}\put(-23.5,26){\small{$5$}}
\put(-23.5,38){\small{$3$}}\put(-17.5,38){\small{$6$}}
\put(-17.5,32){\small{$7$}}
\put(-2,21){\line(1,0){12}}\put(-2,21){\line(0,1){12}}
\put(4,21){\line(0,1){18}}\put(10,21){\line(0,1){24}}
\put(-2,27){\line(1,0){18}}\put(-2,33){\line(1,0){18}}
\put(4,39){\line(1,0){12}}\put(16,39){\line(0,-1){12}}
\put(10,45){\line(1,0){6}}\put(16,45){\line(0,-1){6}}
\put(0.5,23){\small{$2$}}\put(0.5,29){\small{$1$}}
\put(6.5,29){\small{$4$}}\put(6.5,23){\small{$5$}}
\put(6.5,35){\small{$3$}}\put(12.5,35){\small{$7$}}
\put(12.5,41){\small{$6$}}\put(12.5,29){\small{$8$}}
\put(28,25){\line(1,0){12}}\put(28,25){\line(0,1){12}}
\put(34,25){\line(0,1){18}}\put(40,25){\line(0,1){24}}
\put(28,31){\line(1,0){18}}\put(28,37){\line(1,0){18}}
\put(34,43){\line(1,0){12}}\put(46,43){\line(0,-1){12}}
\put(40,49){\line(1,0){6}}\put(46,49){\line(0,-1){6}}
\put(28,25){\line(0,-1){6}}\put(28,19){\line(1,0){6}}
\put(34,19){\line(0,1){6}}
\put(30.5,27){\small{$2$}}\put(30.5,33){\small{$1$}}
\put(30.5,21){\small{$3$}}\put(36.5,33){\small{$5$}}
\put(36.5,39){\small{$4$}}\put(36.5,27){\small{$6$}}
\put(42.5,39){\small{$8$}}\put(42.5,45){\small{$7$}}
\put(42.5,33){\small{$9$}}
\end{picture}
\end{center}
is represented by a natural labeling of its corresponding poset
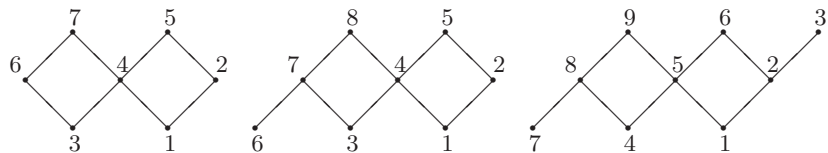
\begin{center}
\setlength{\unitlength}{3pt}
\begin{picture}(40,18)(-5,22)
\put(-32,30){\line(1,1){6}}\put(-32,30){\circle*{0.5}}
\put(-26,36){\circle*{0.5}}\put(-26,36){\line(1,-1){6}}
\put(-20,30){\circle*{0.5}}\put(-20,30){\line(-1,-1){6}}
\put(-26,24){\circle*{0.5}}\put(-26,24){\line(-1,1){6}}
\put(-20,30){\line(1,1){6}}\put(-14,36){\circle*{0.5}}
\put(-14,36){\line(1,-1){6}}\put(-8,30){\circle*{0.5}}
\put(-8,30){\line(-1,-1){6}}\put(-14,24){\circle*{0.5}}
\put(-14,24){\line(-1,1){6}}
\put(-34,31){\small{$6$}}\put(-26.5,37){\small{$7$}}
\put(-14.5,37){\small{$5$}}\put(-20.5,31){\small{$4$}}
\put(-8,31){\small{$2$}}\put(-26.5,21){\small{$3$}}
\put(-14.5,21){\small{$1$}}
\put(3,30){\line(1,1){6}}\put(3,30){\line(-1,-1){6}}
\put(-3,24){\circle*{0.5}}
\put(3,30){\circle*{0.5}}
\put(9,36){\circle*{0.5}}\put(9,36){\line(1,-1){6}}
\put(15,30){\circle*{0.5}}\put(15,30){\line(-1,-1){6}}
\put(9,24){\circle*{0.5}}\put(9,24){\line(-1,1){6}}
\put(15,30){\line(1,1){6}}\put(21,36){\circle*{0.5}}
\put(21,36){\line(1,-1){6}}\put(27,30){\circle*{0.5}}
\put(27,30){\line(-1,-1){6}}\put(21,24){\circle*{0.5}}
\put(21,24){\line(-1,1){6}}
\put(1,31){\small{$7$}}\put(8.5,37){\small{$8$}}
\put(20.5,37){\small{$5$}}\put(14.5,31){\small{$4$}}
\put(27,31){\small{$2$}}\put(8.5,21){\small{$3$}}
\put(20.5,21){\small{$1$}}\put(-3.5,21){\small{$6$}}
\put(38,30){\line(1,1){6}}\put(38,30){\line(-1,-1){6}}
\put(32,24){\circle*{0.5}}
\put(38,30){\circle*{0.5}}
\put(44,36){\circle*{0.5}}\put(44,36){\line(1,-1){6}}
\put(50,30){\circle*{0.5}}\put(50,30){\line(-1,-1){6}}
\put(44,24){\circle*{0.5}}\put(44,24){\line(-1,1){6}}
\put(50,30){\line(1,1){6}}\put(56,36){\circle*{0.5}}
\put(56,36){\line(1,-1){6}}\put(62,30){\circle*{0.5}}
\put(62,30){\line(1,1){6}}\put(68,36){\circle*{0.5}}
\put(62,30){\line(-1,-1){6}}\put(56,24){\circle*{0.5}}
\put(56,24){\line(-1,1){6}}
\put(36,31){\small{$8$}}\put(43.5,37){\small{$9$}}
\put(55.5,37){\small{$6$}}\put(67.5,37){\small{$3$}}
\put(49.5,31){\small{$5$}}
\put(61.5,31){\small{$2$}}\put(43.5,21){\small{$4$}}
\put(55.5,21){\small{$1$}}\put(31.5,21){\small{$7$}}
\end{picture}
\end{center}
from left to right.  Let $\tau_{3n}$ be the Young diagram of shape $(m,m,m,m-1,m-2,\ldots,3)/(m-1,m-2,\ldots,1,0)$ that contains $3n$ boxes when $n\ge 2$, and $\tau_3$ be the Young diagram of shape $(2,2)/(1,0)$. For instance, a standard Young tableau of shape $\tau_{9}$ and its representation are depicted as below.
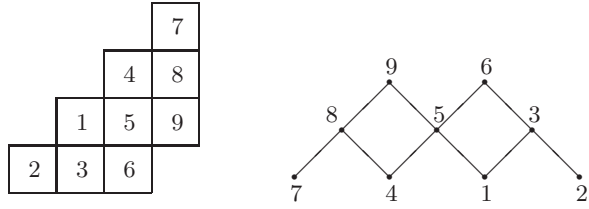
\begin{center}
\setlength{\unitlength}{3pt}
\begin{picture}(30,28)(20,20)
\put(2,22){\line(1,0){12}}\put(2,22){\line(0,1){12}}
\put(8,22){\line(0,1){18}}\put(14,22){\line(0,1){24}}
\put(2,28){\line(1,0){18}}\put(2,34){\line(1,0){18}}
\put(8,40){\line(1,0){12}}\put(20,40){\line(0,-1){12}}
\put(14,46){\line(1,0){6}}\put(20,46){\line(0,-1){6}}
\put(2,22){\line(-1,0){6}}\put(-4,22){\line(0,1){6}}
\put(-4,28){\line(1,0){6}}
\put(4.5,24){\small{$3$}}\put(4.5,30){\small{$1$}}
\put(10.5,30){\small{$5$}}
\put(10.5,36){\small{$4$}}\put(10.5,24){\small{$6$}}
\put(16.5,36){\small{$8$}}\put(16.5,42){\small{$7$}}
\put(16.5,30){\small{$9$}}\put(-1.5,24){\small{$2$}}
\put(38,30){\line(1,1){6}}\put(38,30){\line(-1,-1){6}}
\put(32,24){\circle*{0.5}}
\put(38,30){\circle*{0.5}}
\put(44,36){\circle*{0.5}}\put(44,36){\line(1,-1){6}}
\put(50,30){\circle*{0.5}}\put(50,30){\line(-1,-1){6}}
\put(44,24){\circle*{0.5}}\put(44,24){\line(-1,1){6}}
\put(50,30){\line(1,1){6}}\put(56,36){\circle*{0.5}}
\put(56,36){\line(1,-1){6}}\put(62,30){\circle*{0.5}}
\put(62,30){\line(1,-1){6}}\put(68,24){\circle*{0.5}}
\put(62,30){\line(-1,-1){6}}\put(56,24){\circle*{0.5}}
\put(56,24){\line(-1,1){6}}
\put(36,31){\small{$8$}}\put(43.5,37){\small{$9$}}
\put(55.5,37){\small{$6$}}\put(67.5,21){\small{$2$}}
\put(49.5,31){\small{$5$}}
\put(61.5,31){\small{$3$}}\put(43.5,21){\small{$4$}}
\put(55.5,21){\small{$1$}}\put(31.5,21){\small{$7$}}
\end{picture}
\end{center}
We use $[n]$ to denote the set $\{1,2,\ldots,n\}$. We say a standard Young tableau $T$ of shape $\sigma_{3n-i}$ on a finite subset $I=\{m_1,\ldots,m_{3n-i}\}$ of $\mathbb{N}$ if the corresponding natural labeling of $P_{\sigma_{3n-i}}$ is an order-preserving bijection $\eta:P_{\sigma_{3n-i}}\rightarrow I$. If $I=[3n-i]$, then $T$ is the usual standard Young tableau of shape $\sigma_{3n-i}$. Let $P_{\sigma_{3n-i}^{d}}$ be the dual poset of $P_{\sigma_{3n-i}}$ and $\sigma_{3n-i}^{d}$ be the Hasse diagram of dual poset $P_{\sigma_{3n-i}^{d}}$, then the shape $\sigma_{3n-i}^{d}$ is obtained by flipping the shape $\sigma_{3n-i}$ upside down and therefore $f^{\sigma_{3n-i}}=f^{\sigma_{3n-i}^d}$. We will use this fact to prove two simple but important observations in Lemma~\ref{L:funlem}.
\begin{lemma}\label{L:funlem}
The numbers $f^{\sigma_{3n-2}}$, $f^{\sigma_{3n-1}}$, $f^{\sigma_{3n}}$ and $f^{\tau_{3n}}$ satisfy
\begin{align}\label{E:div1}
(3n-1)f^{\sigma_{3n-2}}&=2f^{\sigma_{3n-1}}\\
\label{E:reTf}(3n)f^{\sigma_{3n-1}}&=f^{\sigma_{3n}}+f^{\tau_{3n}}.
\end{align}
\end{lemma}
\begin{proof}
Given a pair $(T,i)$ where $i\in[3n-1]$ and $T$ is a standard Young tableau of shape $\sigma_{3n-2}$ on the set $[3n-1]-\{i\}$. Suppose $\omega(T)=a_1a_2\cdots a_{3n-2}$ is the reading word of tableau $T$, then the $3$-strip tableau $T$, if we omit the labels
in between, is
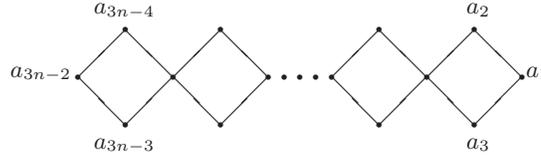
\begin{center}
\setlength{\unitlength}{3pt}
\begin{picture}(-30,18)(10,22)
\put(-32,30){\line(1,1){6}}\put(-32,30){\circle*{0.5}}
\put(-26,36){\circle*{0.5}}\put(-26,36){\line(1,-1){6}}
\put(-20,30){\circle*{0.5}}\put(-20,30){\line(-1,-1){6}}
\put(-26,24){\circle*{0.5}}\put(-26,24){\line(-1,1){6}}
\put(-20,30){\line(1,1){6}}\put(-14,36){\circle*{0.5}}
\put(-14,36){\line(1,-1){6}}\put(-8,30){\circle*{0.5}}
\put(-8,30){\line(-1,-1){6}}\put(-14,24){\circle*{0.5}}
\put(-14,24){\line(-1,1){6}}
\put(-6,30){\circle*{0.5}}\put(-4,30){\circle*{0.5}}
\put(-2,30){\circle*{0.5}}\put(0,30){\circle*{0.5}}
\put(0,30){\line(1,1){6}}\put(6,36){\circle*{0.5}}
\put(6,36){\line(1,-1){6}}\put(12,30){\circle*{0.5}}
\put(12,30){\line(-1,-1){6}}\put(6,24){\circle*{0.5}}
\put(6,24){\line(-1,1){6}}
\put(-40.5,30){\small{$a_{3n-2}$}}
\put(-30,38){\small{$a_{3n-4}$}}
\put(-30,21){\small{$a_{3n-3}$}}
\put(12,30){\line(1,1){6}}\put(18,36){\circle*{0.5}}
\put(18,36){\line(1,-1){6}}\put(24,30){\circle*{0.5}}
\put(24,30){\line(-1,-1){6}}\put(18,24){\circle*{0.5}}
\put(18,24){\line(-1,1){6}}
\put(24.5,30){\small{$a_{1}$}}
\put(17,38){\small{$a_{2}$}}
\put(17,21){\small{$a_{3}$}}
\end{picture}
\end{center}
where $a_j\in [3n-1]-\{i\}$ for all $j$. If $i<a_{3n-2}$, then we let $a_{3n-2}$
cover $i$. Graphically, we obtain
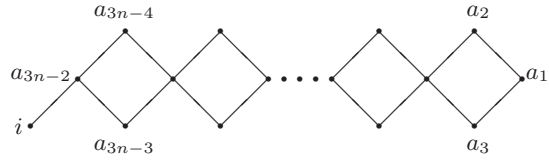
\begin{center}
\setlength{\unitlength}{3pt}
\begin{picture}(-30,18)(10,22)
\put(-32,30){\line(1,1){6}}\put(-32,30){\circle*{0.5}}
\put(-26,36){\circle*{0.5}}\put(-26,36){\line(1,-1){6}}
\put(-20,30){\circle*{0.5}}\put(-20,30){\line(-1,-1){6}}
\put(-26,24){\circle*{0.5}}\put(-26,24){\line(-1,1){6}}
\put(-32,30){\line(-1,-1){6}}\put(-38,24){\circle*{0.5}}
\put(-20,30){\line(1,1){6}}\put(-14,36){\circle*{0.5}}
\put(-14,36){\line(1,-1){6}}\put(-8,30){\circle*{0.5}}
\put(-8,30){\line(-1,-1){6}}\put(-14,24){\circle*{0.5}}
\put(-14,24){\line(-1,1){6}}
\put(-6,30){\circle*{0.5}}\put(-4,30){\circle*{0.5}}
\put(-2,30){\circle*{0.5}}\put(0,30){\circle*{0.5}}
\put(0,30){\line(1,1){6}}\put(6,36){\circle*{0.5}}
\put(6,36){\line(1,-1){6}}\put(12,30){\circle*{0.5}}
\put(12,30){\line(-1,-1){6}}\put(6,24){\circle*{0.5}}
\put(6,24){\line(-1,1){6}}
\put(-40.5,30){\small{$a_{3n-2}$}}
\put(-30,38){\small{$a_{3n-4}$}}
\put(-30,21){\small{$a_{3n-3}$}}
\put(-40,23){\small{$i$}}
\put(12,30){\line(1,1){6}}\put(18,36){\circle*{0.5}}
\put(18,36){\line(1,-1){6}}\put(24,30){\circle*{0.5}}
\put(24,30){\line(-1,-1){6}}\put(18,24){\circle*{0.5}}
\put(18,24){\line(-1,1){6}}
\put(24.5,30){\small{$a_{1}$}}
\put(17,38){\small{$a_{2}$}}
\put(17,21){\small{$a_{3}$}}
\end{picture}
\end{center}
which is a natural labeling on the poset $P_{\sigma_{3n-1}}$. If $i>a_{3n-2}$, then we let $i$ cover $a_{3n-2}$ and obtain
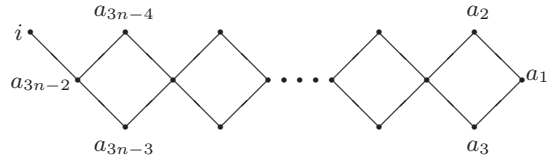
\begin{center}
\setlength{\unitlength}{3pt}
\begin{picture}(-30,18)(10,22)
\put(-32,30){\line(1,1){6}}\put(-32,30){\circle*{0.5}}
\put(-26,36){\circle*{0.5}}\put(-26,36){\line(1,-1){6}}
\put(-20,30){\circle*{0.5}}\put(-20,30){\line(-1,-1){6}}
\put(-26,24){\circle*{0.5}}\put(-26,24){\line(-1,1){6}}
\put(-32,30){\line(-1,1){6}}\put(-38,36){\circle*{0.5}}
\put(-20,30){\line(1,1){6}}\put(-14,36){\circle*{0.5}}
\put(-14,36){\line(1,-1){6}}\put(-8,30){\circle*{0.5}}
\put(-8,30){\line(-1,-1){6}}\put(-14,24){\circle*{0.5}}
\put(-14,24){\line(-1,1){6}}
\put(-6,30){\circle*{0.5}}\put(-4,30){\circle*{0.5}}
\put(-2,30){\circle*{0.5}}\put(0,30){\circle*{0.5}}
\put(0,30){\line(1,1){6}}\put(6,36){\circle*{0.5}}
\put(6,36){\line(1,-1){6}}\put(12,30){\circle*{0.5}}
\put(12,30){\line(-1,-1){6}}\put(6,24){\circle*{0.5}}
\put(6,24){\line(-1,1){6}}
\put(-40.5,29){\small{$a_{3n-2}$}}
\put(-30,38){\small{$a_{3n-4}$}}
\put(-30,21){\small{$a_{3n-3}$}}
\put(-40,35){\small{$i$}}
\put(12,30){\line(1,1){6}}\put(18,36){\circle*{0.5}}
\put(18,36){\line(1,-1){6}}\put(24,30){\circle*{0.5}}
\put(24,30){\line(-1,-1){6}}\put(18,24){\circle*{0.5}}
\put(18,24){\line(-1,1){6}}
\put(24.5,30){\small{$a_{1}$}}
\put(17,38){\small{$a_{2}$}}
\put(17,21){\small{$a_{3}$}}
\end{picture}
\end{center}
which is a natural labeling on the dual poset $P_{\sigma_{3n-1}^{d}}$. It follows that
$(3n-1)f^{\sigma_{3n-2}}=f^{\sigma_{3n-1}}
+f^{\sigma_{3n-1}^{d}}=2f^{\sigma_{3n-1}}$, i.e.,
eq.~(\ref{E:div1}) follows. Next we prove eq.~(\ref{E:reTf}). Given a pair $(T_1,i)$
where $i\in [3n]$ and $T_1$ is a standard Young tableau of shape $\sigma_{3n-1}$ on the set $[3n]-\{i\}$. Suppose $\omega(T_1)=b_1b_2\cdots b_{3n-1}$ is the reading word of tableau $T_1$, if $i<b_1$, then we obtain a tableau of shape $\tau_{3n}$ by letting $b_1$ cover $i$. Otherwise if $i>b_1$, then we obtain a tableau of shape $\sigma_{3n}$ by letting $i$ cover $b_1$. This implies eq.~(\ref{E:reTf}).
\end{proof}
From Lemma~\ref{L:funlem} we find in order to enumerate the $3$-strip tableaux of shape $\sigma_{3n-i}$ for $i=0,1,2$, it suffices to enumerate the standard Young tableaux of shape $\sigma_{3n-2}$ and $\tau_{3n}$. In the following we shall introduce the way to decompose each standard Young tableau of shape $\sigma_{3n-2}$ and $\tau_{3n}$, which gives a combinatorial proof of
\begin{theorem}\label{T:decom}
For $n\ge 2$, the numbers $f^{\sigma_{3n-2}}$ and $f^{\tau_{3n}}$ satisfy
\begin{align}\label{E:rec1}
f^{\sigma_{3n-2}}&=\frac{1}{2n-1}\sum_{i=1}^{n-1}
\binom{3n-2}{3i-1}f^{\sigma_{3i-1}}
f^{\sigma_{3n-3i-1}}\\
\label{E:rec2}
f^{\tau_{3n}}&=\frac{1}{2n+1}\sum_{i=1}^{n-1}\binom{3n}{3i}
f^{\tau_{3i}}f^{\tau_{3n-3i}}.
\end{align}
\end{theorem}
\begin{proof}
For every element $a$ in the standard Young tableau $T_1$ of shape $\sigma_{3n-2}$, there are at most two elements that cover $a$ in the representation of $T_1$. We call an element $b$ the left (resp. right) parent of $a$, denoted by $p_{1,T}(a)$ (resp. $p_{2,T}(a)$), if $b$ covers $a$ and $b$ is to the left (resp. right) of $a$. We next define a reflection $\gamma$ that reverses each standard Young tableau $T_1$ left to right. Below we show the tableau $T_1$ and its mirror image $\gamma(T_1)$.
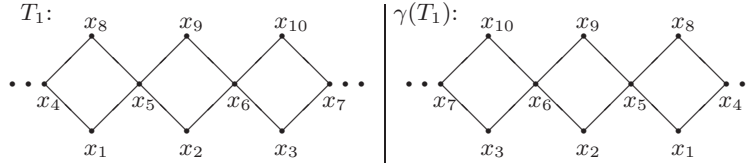
\begin{center}
\setlength{\unitlength}{3pt}
\begin{picture}(10,18)(10,22)
\put(-32,30){\line(1,1){6}}\put(-32,30){\circle*{0.5}}
\put(-26,36){\circle*{0.5}}\put(-26,36){\line(1,-1){6}}
\put(-20,30){\circle*{0.5}}\put(-20,30){\line(-1,-1){6}}
\put(-26,24){\circle*{0.5}}\put(-26,24){\line(-1,1){6}}
\put(-34,30){\circle*{0.5}}\put(-36,30){\circle*{0.5}}
\put(-20,30){\line(1,1){6}}\put(-14,36){\circle*{0.5}}
\put(-14,36){\line(1,-1){6}}\put(-8,30){\circle*{0.5}}
\put(-8,30){\line(-1,-1){6}}\put(-14,24){\circle*{0.5}}
\put(-14,24){\line(-1,1){6}}
\put(-8,30){\line(1,1){6}}\put(-2,36){\circle*{0.5}}
\put(-2,36){\line(1,-1){6}}\put(4,30){\circle*{0.5}}
\put(4,30){\line(-1,-1){6}}\put(-2,24){\circle*{0.5}}
\put(-2,24){\line(-1,1){6}}\put(6,30){\circle*{0.5}}
\put(8,30){\circle*{0.5}}
\put(11,20){\line(0,1){20}}
\put(-27,21){\small{$x_1$}}\put(-15,21){\small{$x_2$}}
\put(-3,21){\small{$x_3$}}
\put(3,27){\small{$x_7$}}
\put(-21,27){\small{$x_5$}}\put(-33,27){\small{$x_4$}}
\put(-9,27){\small{$x_6$}}
\put(-27,37){\small{$x_8$}}\put(-15,37){\small{$x_9$}}
\put(-3,37){\small{$x_{10}$}}\put(-35,38){\small{$T_1$}:}
\put(18,30){\line(1,1){6}}\put(18,30){\circle*{0.5}}
\put(24,36){\circle*{0.5}}\put(24,36){\line(1,-1){6}}
\put(30,30){\circle*{0.5}}\put(30,30){\line(-1,-1){6}}
\put(24,24){\circle*{0.5}}\put(24,24){\line(-1,1){6}}
\put(16,30){\circle*{0.5}}\put(14,30){\circle*{0.5}}
\put(30,30){\line(1,1){6}}\put(36,36){\circle*{0.5}}
\put(36,36){\line(1,-1){6}}\put(42,30){\circle*{0.5}}
\put(42,30){\line(-1,-1){6}}\put(36,24){\circle*{0.5}}
\put(36,24){\line(-1,1){6}}
\put(42,30){\line(1,1){6}}\put(48,36){\circle*{0.5}}
\put(48,36){\line(1,-1){6}}\put(54,30){\circle*{0.5}}
\put(54,30){\line(-1,-1){6}}\put(48,24){\circle*{0.5}}
\put(48,24){\line(-1,1){6}}\put(56,30){\circle*{0.5}}
\put(58,30){\circle*{0.5}}
\put(23,21){\small{$x_3$}}\put(35,21){\small{$x_2$}}
\put(47,21){\small{$x_1$}}
\put(53,27){\small{$x_4$}}
\put(29,27){\small{$x_6$}}\put(17,27){\small{$x_7$}}
\put(41,27){\small{$x_5$}}
\put(23,37){\small{$x_{10}$}}\put(35,37){\small{$x_9$}}
\put(47,37){\small{$x_8$}}\put(12,38){\small{$\gamma(T_1)$}:}
\end{picture}
\end{center}
For $i=0,1,2$, let $\sigma_{3n-i}^{\gamma}$ (resp. $\tau_{3n}^{\gamma}$) be the shape obtained by reversing the shape $\sigma_{3n-i}$ (resp. $\tau_{3n}$) left to right, clearly $\sigma_{3n-2}^{\gamma}=\sigma_{3n-2}$, $\tau_{3n}^{\gamma}=\tau_{3n}$ and $f^{\sigma_{3n-i}^{\gamma}}=f^{\sigma_{3n-i}}$. Let furthermore $\sigma_{3n-i}^{\gamma,d}$ be the shape obtained by reversing the shape $\sigma_{3n-i}^d$ left to right, then $f^{\sigma_{3n-i}^{\gamma,d}}=f^{\sigma_{3n-i}}$. For the tableau $T_1$ of shape $\sigma_{3n-2}$, the tableau $\gamma(T_1)$ has shape $\sigma_{3n-2}^{\gamma}$. The minimal element of $T_1$ is contained in the bottom row, which is the row containing $x_1,x_2,x_3$, see the picture above. We use $\min(T_1)$ to denote the minimal element of tableau $T_1$. For any subset $I=\{x_1,\ldots,x_{3n-2}\}$ of $\mathbb{N}$, let $\CMcal{T}_{\sigma_{3n-2}}^{1,j}(I)$ (resp. $\CMcal{T}_{\sigma_{3n-2}}^{2,j}(I)$) be the set of standard Young tableau $T_1$ on the set $I$ satisfying
\begin{enumerate}
\item the left parent of $\min(T_1)$ is larger (resp. smaller) than the right parent of $\min(T_1)$,
\item $\min(T_1)$ is the $j$-th element on the bottom row of $T_1$ from left to right.
\end{enumerate}
Then from the reflection $\gamma$ we can easily see
\begin{equation}\label{E:T}
\vert\CMcal{T}_{\sigma_{3n-2}}^{1,j}([3n-2])\vert
=\vert\CMcal{T}_{\sigma_{3n-2}}^{2,n-j}([3n-2])\vert\quad\mbox{and}\quad \sum_{j=1}^{n-1}\vert\CMcal{T}_{\sigma_{3n-2}}^{1,j}([3n-2])\vert=\frac{1}{2}f^{\sigma_{3n-2}}.
\end{equation}
Given a pair $(T,i)$ where $i\in [3n-1]$ and $T\in \CMcal{T}_{\sigma_{3n-2}}^{1,j}([3n-1]-\{i\})$, we can represent $T$ as
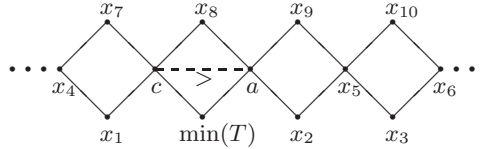
\begin{center}
\setlength{\unitlength}{3pt}
\begin{picture}(40,16)(-25,22)
\put(-32,30){\line(1,1){6}}\put(-32,30){\circle*{0.5}}
\put(-26,36){\circle*{0.5}}\put(-26,36){\line(1,-1){6}}
\put(-20,30){\circle*{0.5}}\put(-20,30){\line(-1,-1){6}}
\put(-26,24){\circle*{0.5}}\put(-26,24){\line(-1,1){6}}
\put(-34,30){\circle*{0.5}}\put(-36,30){\circle*{0.5}}
\put(-38,30){\circle*{0.5}}
\put(-20,30){\line(1,0){1}}\put(-18,30){\line(1,0){1}}
\put(-16,30){\line(1,0){1}}\put(-14,30){\line(1,0){1}}
\put(-12,30){\line(1,0){1}}\put(-10,30){\line(1,0){1}}
\put(-15,28){\small{$>$}}
\put(-20,30){\line(1,1){6}}\put(-14,36){\circle*{0.5}}
\put(-14,36){\line(1,-1){6}}\put(-8,30){\circle*{0.5}}
\put(-8,30){\line(-1,-1){6}}\put(-14,24){\circle*{0.5}}
\put(-14,24){\line(-1,1){6}}
\put(-20.5,27){\small{$c$}}\put(-17,21){\small{$\min(T)$}}
\put(-8.5,27){\small{$a$}}\put(-27,21){\small{$x_1$}}
\put(-3,21){\small{$x_2$}}\put(9,21){\small{$x_3$}}
\put(-33,27){\small{$x_4$}}\put(3,27){\small{$x_5$}}
\put(15,27){\small{$x_6$}}\put(-27,37){\small{$x_7$}}
\put(-15,37){\small{$x_8$}}\put(-3,37){\small{$x_9$}}
\put(9,37){\small{$x_{10}$}}
\put(-8,30){\line(1,1){6}}\put(-2,36){\circle*{0.5}}
\put(-2,36){\line(1,-1){6}}\put(4,30){\circle*{0.5}}
\put(4,30){\line(-1,-1){6}}\put(-2,24){\circle*{0.5}}
\put(-2,24){\line(-1,1){6}}
\put(4,30){\line(1,1){6}}\put(10,36){\circle*{0.5}}
\put(10,36){\line(1,-1){6}}\put(16,30){\circle*{0.5}}
\put(16,30){\line(-1,-1){6}}\put(10,24){\circle*{0.5}}
\put(10,24){\line(-1,1){6}}\put(18,30){\circle*{0.5}}
\put(20,30){\circle*{0.5}}\put(22,30){\circle*{0.5}}
\end{picture}
\end{center}
where $c>a$. Let furthermore $\sigma_{3n-2}^{1,j}$ be the shape of $T$ represented by the Hasse diagram
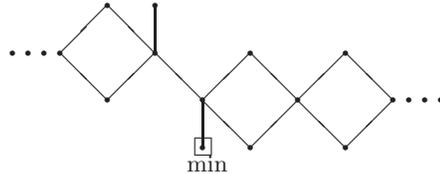
\begin{center}
\setlength{\unitlength}{3pt}
\begin{picture}(40,17)(-25,18)
\put(-32,30){\line(1,1){6}}\put(-32,30){\circle*{0.5}}
\put(-26,36){\circle*{0.5}}\put(-26,36){\line(1,-1){6}}
\put(-20,30){\circle*{0.5}}\put(-20,30){\line(-1,-1){6}}
\put(-26,24){\circle*{0.5}}\put(-26,24){\line(-1,1){6}}
\put(-34,30){\circle*{0.5}}\put(-36,30){\circle*{0.5}}
\put(-38,30){\circle*{0.5}}\put(-20,30){\line(1,-1){6}}
\put(-14,24){\circle*{0.5}}\put(-20,30){\line(0,1){6}}
\put(-20,36){\circle*{0.5}}\put(-14,24){\line(0,-1){6}}
\put(-14,18){\circle*{0.5}}\put(-14,24){\line(1,1){6}}
\put(-14,24){\line(1,-1){6}}\put(-8,30){\circle*{0.5}}
\put(-8,18){\circle*{0.5}}\put(-16,15){\small{$\min$}}
\put(-15.2,17){$\square$}\put(-8,30){\line(1,-1){6}}
\put(-2,24){\circle*{0.5}}\put(-2,24){\line(-1,-1){6}}
\put(-2,24){\line(1,1){6}}\put(4,30){\circle*{0.5}}
\put(4,30){\line(1,-1){6}}\put(10,24){\circle*{0.5}}
\put(10,24){\line(-1,-1){6}}\put(4,18){\circle*{0.5}}
\put(4,18){\line(-1,1){6}}\put(12,24){\circle*{0.5}}
\put(14,24){\circle*{0.5}}\put(16,24){\circle*{0.5}}
\end{picture}
\end{center}
where we use $\square$ to emphasize the location of the minimal element of shape $\sigma_{3n-2}^{1,j}$ is fixed by the assumption on $T$. We will construct the bijection $(T,i)\mapsto g(T,i)$ by comparing the values of $i$ and $a$. If $i>a$, then we let $i$ cover $a$ in the new tableau $g(T,i)$. In this case we notice $\min(g(T,i))=\min(T)=1$ and the tableau $g(T,i)$ can be represented as
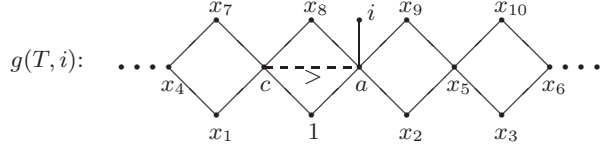
\begin{center}
\setlength{\unitlength}{3pt}
\begin{picture}(40,16)(-25,22)
\put(-32,30){\line(1,1){6}}\put(-32,30){\circle*{0.5}}
\put(-26,36){\circle*{0.5}}\put(-26,36){\line(1,-1){6}}
\put(-20,30){\circle*{0.5}}\put(-20,30){\line(-1,-1){6}}
\put(-26,24){\circle*{0.5}}\put(-26,24){\line(-1,1){6}}
\put(-34,30){\circle*{0.5}}\put(-36,30){\circle*{0.5}}
\put(-38,30){\circle*{0.5}}\put(-52,30){\small{$g(T,i)$}:}
\put(-20,30){\line(1,0){1}}\put(-18,30){\line(1,0){1}}
\put(-16,30){\line(1,0){1}}\put(-14,30){\line(1,0){1}}
\put(-12,30){\line(1,0){1}}\put(-10,30){\line(1,0){1}}
\put(-15,28){\small{$>$}}
\put(-20,30){\line(1,1){6}}\put(-14,36){\circle*{0.5}}
\put(-14,36){\line(1,-1){6}}\put(-8,30){\circle*{0.5}}
\put(-8,30){\line(-1,-1){6}}\put(-14,24){\circle*{0.5}}
\put(-14,24){\line(-1,1){6}}
\put(-20.5,27){\small{$c$}}\put(-14.5,21){\small{$1$}}
\put(-8.5,27){\small{$a$}}\put(-8,30){\line(0,1){6}}
\put(-8,36){\circle*{0.5}}\put(-7,36){\small{$i$}}
\put(-27,21){\small{$x_1$}}
\put(-3,21){\small{$x_2$}}\put(9,21){\small{$x_3$}}
\put(-33,27){\small{$x_4$}}\put(3,27){\small{$x_5$}}
\put(15,27){\small{$x_6$}}\put(-27,37){\small{$x_7$}}
\put(-15,37){\small{$x_8$}}\put(-3,37){\small{$x_9$}}
\put(9,37){\small{$x_{10}$}}
\put(-8,30){\line(1,1){6}}\put(-2,36){\circle*{0.5}}
\put(-2,36){\line(1,-1){6}}\put(4,30){\circle*{0.5}}
\put(4,30){\line(-1,-1){6}}\put(-2,24){\circle*{0.5}}
\put(-2,24){\line(-1,1){6}}
\put(4,30){\line(1,1){6}}\put(10,36){\circle*{0.5}}
\put(10,36){\line(1,-1){6}}\put(16,30){\circle*{0.5}}
\put(16,30){\line(-1,-1){6}}\put(10,24){\circle*{0.5}}
\put(10,24){\line(-1,1){6}}\put(18,30){\circle*{0.5}}
\put(20,30){\circle*{0.5}}\put(22,30){\circle*{0.5}}
\end{picture}
\end{center}
Let $\sigma_{1}^{j}$ be the shape of $g(T,i)$ under the condition $i>a$, represented by the Hasse diagram
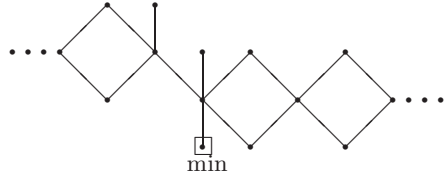
\begin{center}
\setlength{\unitlength}{3pt}
\begin{picture}(40,16)(-25,18)
\put(-32,30){\line(1,1){6}}\put(-32,30){\circle*{0.5}}
\put(-26,36){\circle*{0.5}}\put(-26,36){\line(1,-1){6}}
\put(-20,30){\circle*{0.5}}\put(-20,30){\line(-1,-1){6}}
\put(-26,24){\circle*{0.5}}\put(-26,24){\line(-1,1){6}}
\put(-34,30){\circle*{0.5}}\put(-36,30){\circle*{0.5}}
\put(-38,30){\circle*{0.5}}\put(-20,30){\line(1,-1){6}}
\put(-14,24){\circle*{0.5}}\put(-20,30){\line(0,1){6}}
\put(-20,36){\circle*{0.5}}\put(-14,24){\line(0,-1){6}}
\put(-14,18){\circle*{0.5}}\put(-14,24){\line(1,1){6}}
\put(-14,24){\line(1,-1){6}}\put(-8,30){\circle*{0.5}}
\put(-8,18){\circle*{0.5}}\put(-16,15){\small{$\min$}}
\put(-15.2,17){$\square$}\put(-8,30){\line(1,-1){6}}
\put(-2,24){\circle*{0.5}}\put(-2,24){\line(-1,-1){6}}
\put(-2,24){\line(1,1){6}}\put(4,30){\circle*{0.5}}
\put(4,30){\line(1,-1){6}}\put(10,24){\circle*{0.5}}
\put(10,24){\line(-1,-1){6}}\put(4,18){\circle*{0.5}}
\put(4,18){\line(-1,1){6}}\put(12,24){\circle*{0.5}}
\put(14,24){\circle*{0.5}}\put(16,24){\circle*{0.5}}
\put(-14,24){\line(0,1){6}}\put(-14,30){\circle*{0.5}}
\end{picture}
\end{center}
where the location of the minimal element of shape $\sigma_{1}^{j}$ is fixed by the assumption on $T$. If $i<a$, then we let $a$ cover $i$ in the new tableau $g(T,i)$ and the tableau $g(T,i)$ can be
represented as
\begin{center}
\setlength{\unitlength}{3pt}
\begin{picture}(40,20)(-25,18)
\put(-32,30){\line(1,1){6}}\put(-32,30){\circle*{0.5}}
\put(-26,36){\circle*{0.5}}\put(-26,36){\line(1,-1){6}}
\put(-20,30){\circle*{0.5}}\put(-20,30){\line(-1,-1){6}}
\put(-26,24){\circle*{0.5}}\put(-26,24){\line(-1,1){6}}
\put(-34,30){\circle*{0.5}}\put(-36,30){\circle*{0.5}}
\put(-38,30){\circle*{0.5}}\put(-52,30){\small{$g(T,i)$}:}
\put(-20,30){\line(1,0){1}}\put(-18,30){\line(1,0){1}}
\put(-16,30){\line(1,0){1}}\put(-14,30){\line(1,0){1}}
\put(-12,30){\line(1,0){1}}\put(-10,30){\line(1,0){1}}
\put(-15,31){\small{$>$}}
\put(-20,30){\line(1,1){6}}\put(-14,36){\circle*{0.5}}
\put(-14,36){\line(1,-1){6}}\put(-8,30){\circle*{0.5}}
\put(-8,30){\line(-1,-1){6}}\put(-14,24){\circle*{0.5}}
\put(-14,24){\line(-1,1){6}}
\put(-18,16){Figure $1$}
\put(-20.5,32){\small{$c$}}\put(-18,21){\small{$\min(T)$}}
\put(-8.5,32){\small{$a$}}\put(-8,30){\line(0,-1){6}}
\put(-8,24){\circle*{0.5}}\put(-7,22){\small{$i$}}
\put(-27,21){\small{$x_1$}}
\put(-3,21){\small{$x_2$}}\put(9,21){\small{$x_3$}}
\put(-33,27){\small{$x_4$}}\put(3,27){\small{$x_5$}}
\put(15,27){\small{$x_6$}}\put(-27,37){\small{$x_7$}}
\put(-15,37){\small{$x_8$}}\put(-3,37){\small{$x_9$}}
\put(9,37){\small{$x_{10}$}}
\put(-8,30){\line(1,1){6}}\put(-2,36){\circle*{0.5}}
\put(-2,36){\line(1,-1){6}}\put(4,30){\circle*{0.5}}
\put(4,30){\line(-1,-1){6}}\put(-2,24){\circle*{0.5}}
\put(-2,24){\line(-1,1){6}}
\put(4,30){\line(1,1){6}}\put(10,36){\circle*{0.5}}
\put(10,36){\line(1,-1){6}}\put(16,30){\circle*{0.5}}
\put(16,30){\line(-1,-1){6}}\put(10,24){\circle*{0.5}}
\put(10,24){\line(-1,1){6}}\put(18,30){\circle*{0.5}}
\put(20,30){\circle*{0.5}}\put(22,30){\circle*{0.5}}
\end{picture}
\end{center}
Let $\sigma_{2}^{j}$ be the shape of $g(T,i)$ under the condition $i<a$, represented by
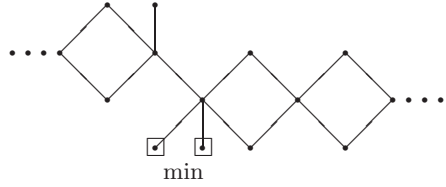
\begin{center}
\setlength{\unitlength}{3pt}
\begin{picture}(40,19)(-25,16)
\put(-32,30){\line(1,1){6}}\put(-32,30){\circle*{0.5}}
\put(-26,36){\circle*{0.5}}\put(-26,36){\line(1,-1){6}}
\put(-20,30){\circle*{0.5}}\put(-20,30){\line(-1,-1){6}}
\put(-26,24){\circle*{0.5}}\put(-26,24){\line(-1,1){6}}
\put(-34,30){\circle*{0.5}}\put(-36,30){\circle*{0.5}}
\put(-38,30){\circle*{0.5}}\put(-20,30){\line(1,-1){6}}
\put(-14,24){\circle*{0.5}}\put(-20,30){\line(0,1){6}}
\put(-20,36){\circle*{0.5}}\put(-14,24){\line(0,-1){6}}
\put(-14,18){\circle*{0.5}}\put(-14,24){\line(1,1){6}}
\put(-14,24){\line(1,-1){6}}\put(-8,30){\circle*{0.5}}
\put(-8,18){\circle*{0.5}}\put(-8,30){\line(1,-1){6}}
\put(-2,24){\circle*{0.5}}\put(-2,24){\line(-1,-1){6}}
\put(-2,24){\line(1,1){6}}\put(4,30){\circle*{0.5}}
\put(4,30){\line(1,-1){6}}\put(10,24){\circle*{0.5}}
\put(10,24){\line(-1,-1){6}}\put(4,18){\circle*{0.5}}
\put(4,18){\line(-1,1){6}}\put(12,24){\circle*{0.5}}
\put(14,24){\circle*{0.5}}\put(16,24){\circle*{0.5}}
\put(-14,24){\line(-1,-1){6}}\put(-20,18){\circle*{0.5}}
\put(-15.2,17){$\square$}\put(-21.2,17){$\square$}
\put(-19,14){\small{$\min$}}
\end{picture}
\end{center}
where only two elements in $\square$ can be the minimal element of the shape $\sigma_{2}^{j}$, namely for the tableau $g(T,i)$ of shape $\sigma_{2}^{j}$, either $\min(g(T,i))=\min(T)$ or $\min(g(T,i))=i$, see Figure $1$. For $i=1,2$, we use
$f^{\sigma_i^{j}}$ to represent the number of natural labelings on the poset $P_{\sigma_i^{j}}$ of shape $\sigma_i^{j}$. We shall next prove
\begin{eqnarray}\label{E:fsig2}
\sum_{j=1}^{n-1}f^{\sigma_1^{j}}=(n-\frac{1}{2})f^{\sigma_{3n-2}}.
\end{eqnarray}
For the tableau $g(T,i)$ of shape $\sigma_{2}^{j}$, either $\min(g(T,i))=\min(T)$ or $\min(g(T,i))=i$, see Figure $1$. If $i<\min(T)$, then $i=\min(g(T,i))=1$ and $\min(T)=2$. By removing $i$ and replacing every element $m$ by $m-1$ from $g(T,i)$, we obtain a standard Young tableau from  $\CMcal{T}_{\sigma_{3n-2}}^{1,j}([3n-2])$, namely the tableau $g(T,i)$ of shape $\sigma_{2}^{j}$ under the condition $i<\min(T)$ is uniquely corresponding to a tableau from $\CMcal{T}_{\sigma_{3n-2}}^{1,j}([3n-2])$. Next we shall show the tableau $g(T,i)$ of shape $\sigma_{2}^{j}$ under the condition $i>\min(T)$ is uniquely corresponding to a tableau from $\CMcal{M}_{\sigma_{3n-2}}^{1,j}$.
Let $\CMcal{M}_{\sigma_{3n-2}}^{1,j}$ (resp. $\CMcal{M}_{\sigma_{3n-2}}^{2,j}$) be the set of standard Young tableau $T_1$ on the set $[n]$ satisfying
\begin{enumerate}
\item the $j$-th element $i'$ on the bottom row of $T_1$ is colored,
\item the left parent of $i'$ is larger (resp. smaller) than the right parent of $i'$.
\end{enumerate}
By the reflection $\gamma$ and the definition, it is clear
\begin{eqnarray}\label{E:M}
\vert \CMcal{M}_{\sigma_{3n-2}}^{1,j}\vert=\vert \CMcal{M}_{\sigma_{3n-2}}^{2,n-j}\vert
\quad\mbox{and}\quad \vert \CMcal{M}_{\sigma_{3n-2}}^{1,j}\vert+\vert \CMcal{M}_{\sigma_{3n-2}}^{2,j}\vert=f^{\sigma_{3n-2}}.
\end{eqnarray}
For the tableau $g(T,i)$ of shape $\sigma_{2}^{j}$, see Figure $1$, if $i>\min(T)$, then $\min(g(T,i))=\min(T)=1$ and $c>a>i$. We remove $\min(T)$ from the tableau $g(T,i)$ and connect $i$ with $c$, next we replace every element $m$ by $m-1$, which gives us
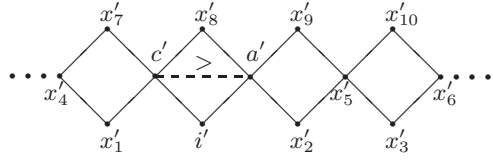
\begin{center}
\setlength{\unitlength}{3pt}
\begin{picture}(40,15)(-25,22)
\put(-32,30){\line(1,1){6}}\put(-32,30){\circle*{0.5}}
\put(-26,36){\circle*{0.5}}\put(-26,36){\line(1,-1){6}}
\put(-20,30){\circle*{0.5}}\put(-20,30){\line(-1,-1){6}}
\put(-26,24){\circle*{0.5}}\put(-26,24){\line(-1,1){6}}
\put(-34,30){\circle*{0.5}}\put(-36,30){\circle*{0.5}}
\put(-38,30){\circle*{0.5}}\put(-20,30){\line(1,0){1}}
\put(-18,30){\line(1,0){1}}\put(-16,30){\line(1,0){1}}
\put(-14,30){\line(1,0){1}}\put(-12,30){\line(1,0){1}}
\put(-10,30){\line(1,0){1}}\put(-15,31){\small{$>$}}
\put(-20,30){\line(1,1){6}}\put(-14,36){\circle*{0.5}}
\put(-14,36){\line(1,-1){6}}\put(-8,30){\circle*{0.5}}
\put(-8,30){\line(-1,-1){6}}\put(-14,24){\circle*{0.5}}
\put(-14,24){\line(-1,1){6}}
\put(-20.5,32){\small{$c'$}}\put(-15,21){\small{$i'$}}
\put(-8.5,32){\small{$a'$}}
\put(-8,30){\line(1,1){6}}\put(-2,36){\circle*{0.5}}
\put(-2,36){\line(1,-1){6}}\put(4,30){\circle*{0.5}}
\put(4,30){\line(-1,-1){6}}\put(-2,24){\circle*{0.5}}
\put(-2,24){\line(-1,1){6}}
\put(4,30){\line(1,1){6}}\put(10,36){\circle*{0.5}}
\put(10,36){\line(1,-1){6}}\put(16,30){\circle*{0.5}}
\put(16,30){\line(-1,-1){6}}\put(10,24){\circle*{0.5}}
\put(10,24){\line(-1,1){6}}\put(18,30){\circle*{0.5}}
\put(20,30){\circle*{0.5}}\put(22,30){\circle*{0.5}}
\put(-27,21){\small{$x_1'$}}
\put(-3,21){\small{$x_2'$}}\put(9,21){\small{$x_3'$}}
\put(-34,27){\small{$x_4'$}}\put(2,27){\small{$x_5'$}}
\put(15,27){\small{$x_6'$}}\put(-27,37){\small{$x_7'$}}
\put(-15,37){\small{$x_8'$}}\put(-3,37){\small{$x_9'$}}
\put(9,37){\small{$x_{10}'$}}
\end{picture}
\end{center}
where $i'=i-1$, $c'=c-1$, $a'=a-1$, $x_m'=x_m-1$ for every $m$ and $i'$ is not necessary to be the minimal element $1$. We color the element $i'$ in the above tableau. Therefore the tableau $g(T,i)$ under the condition $i>\min(T)$, is uniquely corresponding to a tableau from $\CMcal{M}_{\sigma_{3n-2}}^{1,j}$. In combination of these two cases $i<\min(T)$ and $i>\min(T)$, we conclude
\begin{eqnarray*}
f^{\sigma_2^j}=\vert \CMcal{T}_{\sigma_{3n-2}}^{1,j}([3n-2])\vert+\vert\CMcal{M}_{\sigma_{3n-2}}^{1,j}\vert.
\end{eqnarray*}
In view of eq.~(\ref{E:T}) and eq.~(\ref{E:M}), we have
\begin{align*}
\sum_{j=1}^{n-1}f^{\sigma_2^j}
&=\sum_{j=1}^{n-1}\vert \CMcal{T}_{\sigma_{3n-2}}^{1,j}([3n-2])\vert
+\frac{1}{2}\sum_{j=1}^{n-1}(\vert\CMcal{M}_{\sigma_{3n-2}}^{1,j}\vert
+\vert\CMcal{M}_{\sigma_{3n-2}}^{2,n-j}\vert)\\
&=\sum_{j=1}^{n-1}\vert \CMcal{T}_{\sigma_{3n-2}}^{1,j}([3n-2])\vert
+\frac{1}{2}\sum_{j=1}^{n-1}(\vert\CMcal{M}_{\sigma_{3n-2}}^{1,j}\vert
+\vert\CMcal{M}_{\sigma_{3n-2}}^{2,j}\vert)\\
&=\frac{1}{2}f^{\sigma_{3n-2}}+\frac{1}{2}(n-1)f^{\sigma_{3n-2}}=\frac{1}{2}nf^{\sigma_{3n-2}}.
\end{align*}
Furthermore, by inserting $i$ to $T$ where $i\in [3n-1]$ and $T\in\CMcal{T}_{\sigma_{3n-2}}^{1,j}([3n-1]-\{i\})$, we get
$$f^{\sigma_1^{j}}+f^{\sigma_2^{j}}=(3n-1)\vert\CMcal{T}_{\sigma_{3n-2}}^{1,j}([3n-2])\vert,$$
from which it follows
$$\sum_{j=1}^{n-1}(f^{\sigma_1^{j}}+f^{\sigma_2^{j}})=\frac{1}{2}(3n-1)f^{\sigma_{3n-2}}$$
and therefore eq.~(\ref{E:fsig2}) is true. Now it remains to prove
\begin{eqnarray}\label{E:fsig1}
f^{\sigma_1^{j}}+f^{\sigma_1^{n-j}}=\binom{3n-2}{3j-1}f^{\sigma_{3j-1}}f^{\sigma_{3n-3j-1}}.
\end{eqnarray}
For a given tableau $T_{1,1}$ of shape $\sigma_{3j-1}^{\gamma,d}$ and a tableau $T_{1,2}$ of shape $\sigma_{3n-3j-1}^{d}$, we can represent $T_{1,1}$ and $T_{1,2}$ as below:
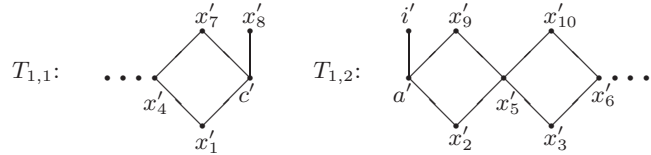
\begin{center}
\setlength{\unitlength}{3pt}
\begin{picture}(40,15)(-25,22.5)
\put(-32,30){\line(1,1){6}}\put(-32,30){\circle*{0.5}}
\put(-26,36){\circle*{0.5}}\put(-26,36){\line(1,-1){6}}
\put(-20,30){\circle*{0.5}}\put(-20,30){\line(-1,-1){6}}
\put(-26,24){\circle*{0.5}}\put(-26,24){\line(-1,1){6}}
\put(-34,30){\circle*{0.5}}\put(-36,30){\circle*{0.5}}
\put(-38,30){\circle*{0.5}}\put(-50,30){\small{$T_{1,1}$}:}
\put(-20,30){\line(0,1){6}}\put(-20,36){\circle*{0.5}}
\put(0,30){\circle*{0.5}}
\put(-33.5,26){\small{$x_4'$}}
\put(-21.5,27){\small{$c'$}}\put(-2,27){\small{$a'$}}
\put(-12,30){\small{$T_{1,2}$}:}
\put(0,30){\line(0,1){6}}\put(0,36){\circle*{0.5}}
\put(-21,37){\small{$x_8'$}}
\put(-27,37){\small{$x_7'$}}
\put(-1,37){\small{$i'$}}\put(-27,21){\small{$x_1'$}}
\put(5,37){\small{$x_9'$}}\put(17,37){\small{$x_{10}'$}}
\put(5,21.5){\small{$x_2'$}}\put(17,21.5){\small{$x_3'$}}
\put(11,26){\small{$x_5'$}}
\put(23,27){\small{$x_6'$}}
\put(0,30){\line(1,1){6}}\put(6,36){\circle*{0.5}}
\put(6,36){\line(1,-1){6}}\put(12,30){\circle*{0.5}}
\put(12,30){\line(-1,-1){6}}\put(6,24){\circle*{0.5}}
\put(6,24){\line(-1,1){6}}
\put(12,30){\line(1,1){6}}\put(18,36){\circle*{0.5}}
\put(18,36){\line(1,-1){6}}\put(24,30){\circle*{0.5}}
\put(24,30){\line(-1,-1){6}}\put(18,24){\circle*{0.5}}
\put(18,24){\line(-1,1){6}}\put(26,30){\circle*{0.5}}
\put(28,30){\circle*{0.5}}\put(30,30){\circle*{0.5}}
\end{picture}
\end{center}
We shall construct a bijection $(T_{1,1},T_{1,2})\mapsto (T_2',T_3')$ where $T_2'$ is a tableau of shape $\sigma_1^j$ and $T_3'$ is a tableau of shape $\sigma_{1}^{n-j}$. If $c'>a'$ between the tableaux $T_{1,1}$ and $T_{1,2}$, then we let $c'$ cover $a'$ and obtain a tableau $T_2$, which is
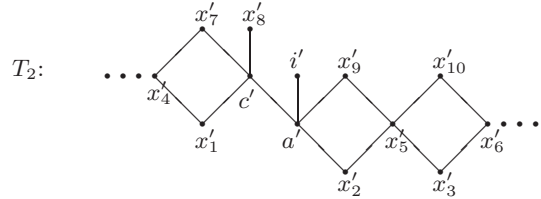
\begin{center}
\setlength{\unitlength}{3pt}
\begin{picture}(40,20)(-25,17.5)
\put(-32,30){\line(1,1){6}}\put(-32,30){\circle*{0.5}}
\put(-26,36){\circle*{0.5}}\put(-26,36){\line(1,-1){6}}
\put(-20,30){\circle*{0.5}}\put(-20,30){\line(-1,-1){6}}
\put(-26,24){\circle*{0.5}}\put(-26,24){\line(-1,1){6}}
\put(-34,30){\circle*{0.5}}\put(-36,30){\circle*{0.5}}
\put(-38,30){\circle*{0.5}}\put(-50,30){\small{$T_2$}:}
\put(-20,30){\line(0,1){6}}\put(-20,36){\circle*{0.5}}
\put(-20,30){\line(1,-1){6}}\put(-14,24){\circle*{0.5}}
\put(-21.5,26){\small{$c'$}}\put(-16,21){\small{$a'$}}
\put(-14,24){\line(0,1){6}}\put(-14,30){\circle*{0.5}}
\put(-21,37){\small{$x_8'$}}
\put(-15,31){\small{$i'$}}
\put(-14,24){\line(1,1){6}}\put(-8,30){\circle*{0.5}}
\put(-8,30){\line(1,-1){6}}\put(-2,24){\circle*{0.5}}
\put(-2,24){\line(-1,-1){6}}\put(-8,18){\circle*{0.5}}
\put(-8,18){\line(-1,1){6}}
\put(-2,24){\line(1,1){6}}\put(4,30){\circle*{0.5}}
\put(4,30){\line(1,-1){6}}\put(10,24){\circle*{0.5}}
\put(10,24){\line(-1,-1){6}}\put(4,18){\circle*{0.5}}
\put(4,18){\line(-1,1){6}}\put(12,24){\circle*{0.5}}
\put(14,24){\circle*{0.5}}\put(16,24){\circle*{0.5}}
\put(-27,21){\small{$x_1'$}}
\put(-9,15.5){\small{$x_2'$}}\put(3,15.5){\small{$x_3'$}}
\put(-33,27){\small{$x_4'$}}\put(-3,21){\small{$x_5'$}}
\put(9,21){\small{$x_6'$}}\put(-27,37){\small{$x_7'$}}
\put(-9,31){\small{$x_9'$}}
\put(3,31){\small{$x_{10}'$}}
\end{picture}
\end{center}
The tableau $T_2$ is uniquely corresponding to a tableau of shape $\sigma_{1}^{j}$. Let $T_2'$ be the tableau obtained from $T_2$ by replacing every element $m$ in $T_2$ by $m+1$, and then let the element $(a'+1)$ cover $1$ in the diagram. Clearly $T_2'$ is a tableau of shape $\sigma_{1}^{j}$ on the set $[3n-2]$.

If $c'<a'$ between the tableaux $T_{1,1}$ and $T_{1,2}$, then we let $a'$ cover $c'$ and reverse the whole diagram left to right. This gives us a tableau $T_{3}$, which is
\begin{center}
\setlength{\unitlength}{3pt}
\begin{picture}(80,19.5)(-25,12)
\put(-16,21){\small{$x_6'$}}\put(-14,24){\circle*{0.5}}
\put(10,24){\line(0,1){6}}\put(10,30){\circle*{0.5}}
\put(-16,24){\circle*{0.5}}\put(-18,24){\circle*{0.5}}
\put(-20,24){\circle*{0.5}}\put(-30,24){\small{$T_{3}$}:}
\put(-14,24){\line(1,1){6}}\put(-8,30){\circle*{0.5}}
\put(-8,30){\line(1,-1){6}}\put(-2,24){\circle*{0.5}}
\put(-2,24){\line(-1,-1){6}}\put(-8,18){\circle*{0.5}}
\put(-8,18){\line(-1,1){6}}
\put(-2,24){\line(1,1){6}}\put(4,30){\circle*{0.5}}
\put(4,30){\line(1,-1){6}}\put(10,24){\circle*{0.5}}
\put(10,24){\line(-1,-1){6}}\put(4,18){\circle*{0.5}}
\put(4,18){\line(-1,1){6}}
\put(-9,15.5){\small{$x_3'$}}\put(3,15.5){\small{$x_2'$}}
\put(-3,21){\small{$x_5'$}}
\put(9,21){\small{$a'$}}\put(-9,31){\small{$x_{10}'$}}
\put(3,31){\small{$x_{9}'$}}\put(9,31){\small{$i'$}}
\put(10,24){\line(1,-1){6}}\put(16,18){\circle*{0.5}}
\put(16,18){\line(1,1){6}}\put(22,24){\circle*{0.5}}
\put(22,24){\line(1,-1){6}}\put(28,18){\circle*{0.5}}
\put(28,18){\line(-1,-1){6}}\put(22,12){\circle*{0.5}}
\put(22,12){\line(-1,1){6}}\put(16,18){\line(0,1){6}}
\put(16,24){\circle*{0.5}}\put(15,15){\small{$c'$}}
\put(15,25){\small{$x_8'$}}\put(21,25){\small{$x_7'$}}
\put(27,15){\small{$x_4'$}}\put(21,10){\small{$x_1'$}}
\put(30,18){\circle*{0.5}}\put(32,18){\circle*{0.5}}
\put(34,18){\circle*{0.5}}
\end{picture}
\end{center}
The tableau $T_{3}$ is uniquely corresponding to a tableau of shape $\sigma_{1}^{n-j}$. Let $T_3'$ be the tableau obtained from $T_3$ by replacing every element $m$ in $T_3$ by $m+1$, and then let the element $(c'+1)$ cover $1$ in the diagram. Clearly $T_3'$ is a tableau of shape $\sigma_{1}^{n-j}$ on the set $[3n-2]$. Note that the above process $(T_{1,1},T_{1,2})\mapsto (T_2',T_3')$ is invertible. In view of
$f^{\sigma_{3j-1}^{\gamma,d}}=f^{\sigma_{3j-1}}$ and $f^{\sigma_{3n-3j-1}^d}=f^{\sigma_{3n-3j-1}}$, eq.~(\ref{E:fsig1}) follows. In combination of eq.~(\ref{E:fsig1}), eq.~(\ref{E:fsig2}), we get eq.~(\ref{E:rec1})
immediately. The proof of eq.~(\ref{E:rec2}) follows analogously
to eq.~(\ref{E:rec1}). The only difference is that we need to consider the tableaux of shape $\tau_{3n}^d$ instead of $\tau_{3n}$. We notice that eq.~(\ref{E:rec2}) is equivalent to
\begin{eqnarray}\label{E:rec2md}
f^{\tau_{3n}^d}=\frac{1}{2n+1}\sum_{i=1}^{n-1}\binom{3n}{3i}
f^{\tau_{3i}^d}f^{\tau_{3n-3i}^d}.
\end{eqnarray}
For an integer $i\in [3n+1]$ and a tableau $T$ of shape $\tau_{3n}^{d}$ on the set $[3n+1]-\{i\}$ as below
\begin{center}
\setlength{\unitlength}{3pt}
\begin{picture}(40,16)(-25,22)
\put(-32,30){\line(1,1){6}}\put(-32,30){\circle*{0.5}}
\put(-26,36){\circle*{0.5}}\put(-26,36){\line(1,-1){6}}
\put(-20,30){\circle*{0.5}}\put(-20,30){\line(-1,-1){6}}
\put(-26,24){\circle*{0.5}}\put(-26,24){\line(-1,1){6}}
\put(-32,30){\line(-1,1){6}}\put(-38,36){\circle*{0.5}}
\put(-20,30){\line(1,0){1}}\put(-18,30){\line(1,0){1}}
\put(-16,30){\line(1,0){1}}\put(-14,30){\line(1,0){1}}
\put(-12,30){\line(1,0){1}}\put(-10,30){\line(1,0){1}}
\put(-15,28){\small{$>$}}\put(-48,30){\small{$T$}:}
\put(-20,30){\line(1,1){6}}\put(-14,36){\circle*{0.5}}
\put(-14,36){\line(1,-1){6}}\put(-8,30){\circle*{0.5}}
\put(-8,30){\line(-1,-1){6}}\put(-14,24){\circle*{0.5}}
\put(-14,24){\line(-1,1){6}}\put(-6,30){\circle*{0.5}}
\put(-4,30){\circle*{0.5}}\put(-2,30){\circle*{0.5}}
\put(-20.5,27){\small{$c$}}\put(-17,21){\small{$\min(T)$}}
\put(-8.5,27){\small{$a$}}\put(-27,21){\small{$x_1$}}
\put(3,21){\small{$x_2$}}\put(15,21){\small{$x_3$}}
\put(-33,27){\small{$x_4$}}\put(9,27){\small{$x_6$}}
\put(21,27){\small{$x_7$}}\put(-4,27){\small{$x_{5}$}}
\put(-27,37){\small{$x_8$}}\put(-39,37){\small{$x_{12}$}}
\put(-15,37){\small{$x_9$}}\put(3,37){\small{$x_{10}$}}
\put(15,37){\small{$x_{11}$}}\put(27,37){\small{$x_{13}$}}
\put(-2,30){\line(1,1){6}}\put(4,36){\circle*{0.5}}
\put(4,36){\line(1,-1){6}}\put(10,30){\circle*{0.5}}
\put(10,30){\line(-1,-1){6}}\put(4,24){\circle*{0.5}}
\put(4,24){\line(-1,1){6}}
\put(10,30){\line(1,1){6}}\put(16,36){\circle*{0.5}}
\put(16,36){\line(1,-1){6}}\put(22,30){\circle*{0.5}}
\put(22,30){\line(-1,-1){6}}\put(16,24){\circle*{0.5}}
\put(16,24){\line(-1,1){6}}\put(22,30){\line(1,1){6}}
\put(28,36){\circle*{0.5}}
\end{picture}
\end{center}
where $c>a$. We insert $i$ into $T$ by comparing $i$ and $a$. For simplicity, we choose the tableau $T$ of shape $\tau_{3n}^d$ over $\tau_{3n}$ since the minimal element of $T$ is contained in the bottom row, where every element has both left parent and right parent. The discussion for the location of the minimal element in $T$ follows the same to that for the shape $\sigma_{3n-2}$. But for a tableau $\bar{T}$ of shape $\tau_{3n}$, we need to further discuss the case when the first element or the last element on the bottom row of $\bar{T}$ is minimal, i.e., the case
\begin{center}
\setlength{\unitlength}{3pt}
\begin{picture}(40,16)(-25,22)
\put(-32,30){\line(1,1){6}}\put(-32,30){\circle*{0.5}}
\put(-26,36){\circle*{0.5}}\put(-26,36){\line(1,-1){6}}
\put(-20,30){\circle*{0.5}}\put(-20,30){\line(-1,-1){6}}
\put(-26,24){\circle*{0.5}}\put(-26,24){\line(-1,1){6}}
\put(-32,30){\line(-1,-1){6}}\put(-38,24){\circle*{0.5}}
\put(-48,30){\small{$\bar{T}$}:}
\put(-20,30){\line(1,1){6}}\put(-14,36){\circle*{0.5}}
\put(-14,36){\line(1,-1){6}}\put(-8,30){\circle*{0.5}}
\put(-8,30){\line(-1,-1){6}}\put(-14,24){\circle*{0.5}}
\put(-14,24){\line(-1,1){6}}\put(-6,30){\circle*{0.5}}
\put(-4,30){\circle*{0.5}}\put(-2,30){\circle*{0.5}}
\put(-20.5,27){\small{$\bar{c}$}}\put(-15,21){\small{$\bar{x}_0$}}
\put(-8.5,27){\small{$\bar{a}$}}\put(-27,21){\small{$\bar{x}_1$}}
\put(3,21){\small{$\bar{x}_2$}}\put(15,21){\small{$\bar{x}_3$}}
\put(-33,27){\small{$\bar{x}_4$}}\put(9,27){\small{$\bar{x}_6$}}
\put(21,27){\small{$\bar{x}_7$}}\put(-4,27){\small{$\bar{x}_{5}$}}
\put(-27,37){\small{$\bar{x}_8$}}\put(-40,21){\small{$\bar{x}_{12}$}}
\put(-15,37){\small{$\bar{x}_9$}}\put(3,37){\small{$\bar{x}_{10}$}}
\put(15,37){\small{$\bar{x}_{11}$}}\put(27,21){\small{$\bar{x}_{13}$}}
\put(-2,30){\line(1,1){6}}\put(4,36){\circle*{0.5}}
\put(4,36){\line(1,-1){6}}\put(10,30){\circle*{0.5}}
\put(10,30){\line(-1,-1){6}}\put(4,24){\circle*{0.5}}
\put(4,24){\line(-1,1){6}}
\put(10,30){\line(1,1){6}}\put(16,36){\circle*{0.5}}
\put(16,36){\line(1,-1){6}}\put(22,30){\circle*{0.5}}
\put(22,30){\line(-1,-1){6}}\put(16,24){\circle*{0.5}}
\put(16,24){\line(-1,1){6}}\put(22,30){\line(1,-1){6}}
\put(28,24){\circle*{0.5}}
\end{picture}
\end{center}
when $\bar{x}_{12}$ or $\bar{x}_{13}$ is $\min(\bar{T})$. The rest of the proof for eq.~(\ref{E:rec2md})
follows the same to that for eq.~(\ref{E:rec1}) and is omitted here.
\end{proof}

{\bf Remark}: In the proof of Theorem~\ref{T:decom} we use the property that both the shape $\sigma_{3n-2}$ and $\tau_{3n}$ are preserved under the reflection $\gamma$. This condition is necessary for the decomposition of shape $\sigma_{3n-2}$ because of eq.~(\ref{E:fsig1}), where we need to reverse the diagram left to right. However, the ``decomposition'' idea shown in Theorem~\ref{T:decom} is not restricted to the shape that is preserved under the reflection $\gamma$. For instance, we can decompose every $3$-strip tableau of shape $\sigma_{3n-1}^d$ even though the shape $\sigma_{3n-1}^d$ is not preserved under the reflection $\gamma$. Let $\CMcal{C}_n$ be the set of $3$-strip tableaux of shape $\sigma_{3n-1}^d$ and $\sigma_{3n-1}^{\gamma,d}$, then the set $\CMcal{C}_n$ is closed under the reflection $\gamma$. We can apply the proof of Theorem~\ref{T:decom} on any tableau from $\CMcal{C}_n$, which yields: For $n\ge 2$,
\begin{align*}
f^{\sigma_{3n-1}}=\frac{1}{2n}\sum_{i=1}^{n-1}\binom{3n-1}{3i}f^{\tau_{3i}}f^{\sigma_{3n-3i-1}}.
\end{align*}

From Theorem~\ref{T:decom} we can prove Theorem~\ref{T:3strip} by using the generating function approach.
\subsection{Proof of Theorem~\ref{T:3strip}}\label{S:sub1}
This subsection is independent of Section~\ref{S:bij}. We first derive the generating functions for $f^{\sigma_{3n-1}}$ and $f^{\tau_{3n}}$ from the recursions given in Theorem~\ref{T:decom}. In combination of eq.~(\ref{E:reTf}), we can further prove the generating functions for $f^{\sigma_{3n}}$ given in eq.~(\ref{E:gen3}). In view of eq.~(\ref{E:div1}), we immediately have the generating function of
$f^{\sigma_{3n-2}}$ after we prove eq.~(\ref{E:gen2}). By extracting the coefficients from these generating functions, Theorem~\ref{T:3strip}
is proved. More precisely, let
\begin{eqnarray*}
f(x)=\sum_{n\ge 1}\frac{f^{\sigma_{3n-2}}x^{2n-1}}{(3n-2)!},\quad
g(x)=\sum_{n\ge 1}\frac{f^{\sigma_{3n-1}}x^{2n-1}}{(3n-1)!},\quad
h(x)=\sum_{n\ge 1}\frac{f^{\tau_{3n}}x^{2n-1}}{(3n)!},
\end{eqnarray*}
then from eq.~(\ref{E:div1}) we have $f(x)=2g(x)$. Furthermore, eq.~(\ref{E:rec1}) is
equivalent to
\begin{equation*}
f'(x)=1+g(x)^2\quad \mbox{ where } f(0)=0.
\end{equation*}
This leads to a unique solution,
$g(x)=\tan(x/2)$. Together with the exponential generating function for $E_{2n-1}$,
eq.~(\ref{E:3s2}) and eq.~(\ref{E:3s1}) are proved. Similarly, eq.~(\ref{E:rec2}) is equivalent to
\begin{equation*}
-xh'(x)=-x+2h(x)-xh^2(x)\quad \mbox{ where } h(0)=0.
\end{equation*}
This yields a unique solution
$$h(x)=-\frac{1}{\tan(x)}+\frac{1}{x}=\frac{1}{3}x+\frac{1}{45}x^3+\cdots$$
and consequently in view of eq.~(\ref{E:reTf}), the exponential generating function for $f^{\sigma_{3n}}$ is equal to $g(x)-h(x)$, thus eq.~(\ref{E:gen3}) follows. By considering the expansion of $x/\sin(x)$, we finally obtain the coefficients $f^{\sigma_{3n}}$, i.e., eq.~(\ref{E:3s}). Alternatively, we can obtain the expression of $f^{\sigma_{3n-2}}$ and $f^{\tau_{3n}}$ by using the recursions of tangent numbers $E_{2n-1}$ and Bernoulli numbers $B_{2n}$. The Bernoulli numbers $B_{2n}$ are integers defined from the tangent numbers $E_{2n-1}$ by the relation:
\begin{eqnarray*}
B_{2n}=\frac{nE_{2n-1}}{2^{2n-1}(2^{2n}-1)}\quad \mbox{where}\quad n\ge 1.
\end{eqnarray*}
The recursions for $E_{2n-1}$ and $B_{2n}$ are
\begin{align*}
E_{2n-1}=\sum_{i=1}^{n-1}\binom{2n-2}{2i-1}E_{2i-1}E_{2n-2i-1},\quad
B_{2n}=\frac{1}{2n+1}\sum_{i=1}^{n-1}\binom{2n}{2i}B_{2i}B_{2n-2i}.
\end{align*}
After verifying the initial conditions that $f^{\sigma_1}=E_1$, $2^2f^{\sigma_4}=4E_3$,
$1!(2^2-1)f^{\tau_3}=3!E_1$ and $3!(2^4-1)f^{\tau_6}=6!E_3$, we can inductively prove eq.~(\ref{E:3s2}) and
$$f^{\tau_{3n}}=\frac{(3n)!2^{2n}B_{2n}}{(2n)!}=\frac{(3n)!E_{2n-1}}{(2n-1)!(2^{2n}-1)}.$$
In view of Lemma~\ref{L:funlem}, we can further obtain the expression of $f^{\sigma_{3n-1}}$ and $f^{\sigma_{3n}}$. Henceforth the proof of Theorem~\ref{T:3strip} is complete.
\qed
\section{Bijective Proof of Theorem~\ref{T:3strip}}\label{S:bij}
Recall that from Lemma~\ref{L:funlem} we find it is sufficient to prove
\begin{align}\label{E:bij1}
2^{2n-2}f^{\sigma_{3n-2}}=&\frac{(3n-2)!}{(2n-1)!}E_{2n-1},\\
\label{E:taubij}(2^{2n}-1)f^{\tau_{3n}}=&\frac{(3n)!}{(2n-1)!}E_{2n-1}.
\end{align}
We will continue using the natural labeling on the corresponding poset to represent the standard Young  tableaux of shape $\sigma_{3n-2}$ and $\tau_{3n}$. We shall first prove eq.~(\ref{E:bij1}). We call $\sigma=\sigma_1\sigma_2\ldots \sigma_{2n-1}$ an up-down permutation (resp. a down-up permutation) on the set $\{\sigma_1,\sigma_2,\ldots,\sigma_{2n-1}\}$ of $\mathbb{N}$ if $\sigma_1<\sigma_2>\cdots <\sigma_{2n-2}>\sigma_{2n-1}$ (resp. $\sigma_1>\sigma_2<\cdots >\sigma_{2n-2}<\sigma_{2n-1}$) and $\sigma_1,\sigma_2,\ldots,\sigma_{2n-1}$ are $(2n-1)$ distinct positive integers. For $n\ge 2$, let $\CMcal{A}_n$ be the set of pairs $(\sigma,(x_1,\ldots,x_{n-1}))$ where $(x_1,\ldots,x_{n-1})$ is a sequence of $(n-1)$ distinct positive integers from $[3n-2]$, and $\sigma=\sigma_1\cdots \sigma_{2n-1}$ is an up-down permutation on the set $[3n-2]-\{x_1,\ldots,x_{n-1}\}$, let furthermore $\CMcal{A}_1=\{\cdot _1\}$ be the set of a single point labeled with $1$. It is clear that the set $\CMcal{A}_n$ has cardinality $\vert\CMcal{A}_n\vert=(3n-2)!/(2n-1)!E_{2n-1}$ for $n\ge 1$.
If we omit the labels in between, $\sigma$ is represented as
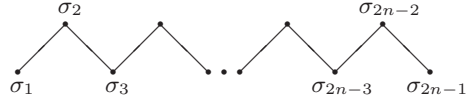
\begin{center}
\setlength{\unitlength}{3pt}
\begin{picture}(40,15)(-25,25)
\put(-32,30){\line(1,1){6}}\put(-32,30){\circle*{0.5}}
\put(-26,36){\circle*{0.5}}\put(-26,36){\line(1,-1){6}}
\put(-20,30){\circle*{0.5}}\put(-20,30){\line(1,1){6}}
\put(-14,36){\circle*{0.5}}\put(-14,36){\line(1,-1){6}}
\put(-8,30){\circle*{0.5}}\put(-6,30){\circle*{0.5}}
\put(-4,30){\circle*{0.5}}
\put(-4,30){\line(1,1){6}}
\put(2,36){\circle*{0.5}}\put(2,36){\line(1,-1){6}}
\put(8,30){\circle*{0.5}}\put(8,30){\line(1,1){6}}
\put(14,36){\circle*{0.5}}\put(14,36){\line(1,-1){6}}
\put(20,30){\circle*{0.5}}
\put(-33,27.5){\small{$\sigma_1$}}\put(-21,27.5){\small{$\sigma_3$}}
\put(-27,37.5){\small{$\sigma_2$}}
\put(5,27.5){\small{$\sigma_{2n-3}$}}\put(17,27.5){\small{$\sigma_{2n-1}$}}
\put(11,37.5){\small{$\sigma_{2n-2}$}}
\end{picture}
\end{center}
For $n\ge 2$, we will add the integers $x_1,\ldots,x_{n-1}$ successively into $\sigma$ as follows:
\begin{tabbing}
{\bf Algorithm $1$}: {Insert the sequence $(x_1,\ldots,x_{n-1})$ into the up-down permutation $\sigma$}.\\
\rule{1cm}{0mm}for $i:=1\rightarrow n-1$ do\\
\rule{2cm}{0mm}\=$r_i:=\min\{\sigma_{2i-1},\sigma_{2i+1}\}$;\\
\>if $x_i<r_i$ then\\
\>\rule{1cm}{0mm}\=let both $\sigma_{2i-1}$ and $\sigma_{2i+1}$ cover $x_i$;\\
\>else \\
\>\rule{1cm}{0mm}\=let $x_i$ cover $r_i$ and remove the edge $\{\sigma_{2i},r_i\}$;\\
\>end if\\
\rule{1cm}{0mm}end do
\end{tabbing}
In the Algorithm $1$, we actually apply the idea of decomposing $3$-strip tableaux of shape $\sigma_{3n-2}$ to the up-down permutation $\sigma$, namely each time when $x_i>r_i$ we have
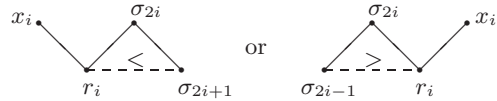
\begin{center}
\setlength{\unitlength}{3pt}
\begin{picture}(40,15)(-35,25)
\put(-37,27){\line(1,1){6}}\put(-37,27){\circle*{0.5}}
\put(-31,33){\circle*{0.5}}\put(-31,33){\line(1,-1){6}}
\put(-25,27){\circle*{0.5}}\put(-37,27){\line(1,0){1}}
\put(-35,27){\line(1,0){1}}\put(-33,27){\line(1,0){1}}
\put(-31,27){\line(1,0){1}}\put(-29,27){\line(1,0){1}}
\put(-27,27){\line(1,0){1}}\put(-32,27.5){\small{$<$}}
\put(-37,27){\line(-1,1){6}}\put(-43,33){\circle*{0.5}}
\put(-17,29){\small{or}}\put(-7,27){\circle*{0.5}}
\put(-7,27){\line(1,1){6}}\put(-1,33){\circle*{0.5}}
\put(-1,33){\line(1,-1){6}}\put(5,27){\circle*{0.5}}
\put(-7,27){\line(1,0){1}}\put(-5,27){\line(1,0){1}}
\put(-3,27){\line(1,0){1}}\put(-1,27){\line(1,0){1}}
\put(1,27){\line(1,0){1}}\put(3,27){\line(1,0){1}}
\put(-2,27.5){\small{$>$}}\put(5,27){\line(1,1){6}}
\put(11,33){\circle*{0.5}}
\put(-37.5,24){\small{$r_i$}}\put(-25.5,24){\small{$\sigma_{2i+1}$}}
\put(-46.5,33){\small{$x_i$}}\put(-10,24){\small{$\sigma_{2i-1}$}}
\put(4.5,24){\small{$r_i$}}\put(12,33){\small{$x_i$}}
\put(-31.5,34){\small{$\sigma_{2i}$}}
\put(-1.5,34){\small{$\sigma_{2i}$}}
\end{picture}
\end{center}
Since $\sigma_{2i}>\sigma_{2i+1}>r_i$ on the left and $\sigma_{2i}>\sigma_{2i-1}>r_i$ on the right, the above diagrams are equivalent to
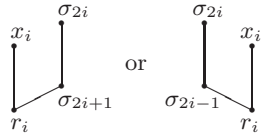
\begin{center}
\setlength{\unitlength}{3pt}
\begin{picture}(40,15)(-35,25)
\put(-31,27){\circle*{0.5}}\put(-31,27){\line(0,1){8}}
\put(-31,35){\circle*{0.5}}\put(-31,27){\line(2,1){6}}
\put(-25,30){\circle*{0.5}}\put(-25,30){\line(0,1){8}}
\put(-25,38){\circle*{0.5}}\put(-31.5,24.5){\small{$r_i$}}
\put(-31.5,36){\small{$x_i$}}
\put(-25.5,39){\small{$\sigma_{2i}$}}
\put(-25.5,27.5){\small{$\sigma_{2i+1}$}}
\put(-12,27.5){\small{$\sigma_{2i-1}$}}
\put(-17,32){\small{or}}
\put(-7,30){\circle*{0.5}}\put(-7,30){\line(0,1){8}}
\put(-7,38){\circle*{0.5}}\put(-7,30){\line(2,-1){6}}
\put(-1,27){\circle*{0.5}}\put(-1,27){\line(0,1){8}}
\put(-1,35){\circle*{0.5}}\put(-2.5,24.5){\small{$r_i$}}
\put(-2.5,36){\small{$x_i$}}
\put(-8.5,39){\small{$\sigma_{2i}$}}
\end{picture}
\end{center}
This allows us to separate them into two independent pieces by removing the edge $\{r_i,\sigma_{2i+1}\}$ on the left and the edge $\{r_i,\sigma_{2i-1}\}$ on the right. In fact, the above process is invertible and henceforth Algorithm $1$ is a bijection between $\CMcal{A}_n$ and its image set. For example, consider the pair $(\bar{\sigma}_1,(3,6,1,8))$ where $n=5$, $\bar{\sigma}_1=25497\,11\,10\,13\,12$ is an up-down permutation and $(3,6,1,8)$ is a sequence of integers, i.e., $\bar{\sigma}_1$ is
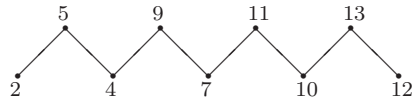
\begin{center}
\setlength{\unitlength}{3pt}
\begin{picture}(40,15)(-25,25)
\put(-32,30){\line(1,1){6}}\put(-32,30){\circle*{0.5}}
\put(-26,36){\circle*{0.5}}\put(-26,36){\line(1,-1){6}}
\put(-20,30){\circle*{0.5}}\put(-20,30){\line(1,1){6}}
\put(-14,36){\circle*{0.5}}\put(-14,36){\line(1,-1){6}}
\put(-8,30){\circle*{0.5}}\put(-8,30){\line(1,1){6}}
\put(-2,36){\circle*{0.5}}\put(-2,36){\line(1,-1){6}}
\put(4,30){\circle*{0.5}}\put(4,30){\line(1,1){6}}
\put(10,36){\circle*{0.5}}\put(10,36){\line(1,-1){6}}
\put(16,30){\circle*{0.5}}
\put(-33,27.5){\footnotesize{$2$}}\put(-21,27.5){\footnotesize{$4$}}
\put(-27,37){\footnotesize{$5$}}\put(-15,37){\footnotesize{$9$}}\put(-3,37){\footnotesize{$11$}}
\put(-9,27.5){\footnotesize{$7$}}
\put(3,27.5){\footnotesize{$10$}}\put(15,27.5){\footnotesize{$12$}}
\put(9,37){\footnotesize{$13$}}
\end{picture}
\end{center}
After applying Algorithm $1$ on the pair $(\bar{\sigma}_1,(3,6,1,8))$, we finally obtain
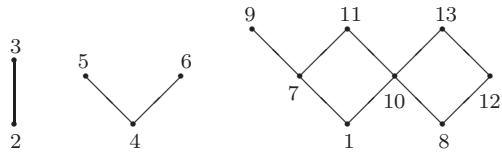
\begin{center}
\setlength{\unitlength}{3pt}
\begin{picture}(40,18)(-25,24)
\put(-37,27){\line(0,1){8}}\put(-37,27){\circle*{0.5}}
\put(-37,35){\circle*{0.5}}
\put(-37.5,24){\footnotesize{$2$}}\put(-37.5,36){\footnotesize{$3$}}
\put(-22,27){\circle*{0.5}}
\put(-22,27){\line(-1,1){6}}\put(-22,27){\line(1,1){6}}
\put(-28,33){\circle*{0.5}}\put(-22.5,24){\footnotesize{$4$}}
\put(-16,33){\circle*{0.5}}\put(-29,34){\footnotesize{$5$}}
\put(-16,34){\footnotesize{$6$}}
\put(5,27){\circle*{0.5}}\put(5,27){\line(-1,1){6}}
\put(-1,33){\circle*{0.5}}\put(-1,33){\line(-1,1){6}}
\put(-7,39){\circle*{0.5}}\put(5,27){\line(1,1){6}}
\put(11,33){\circle*{0.5}}\put(11,33){\line(-1,1){6}}
\put(5,39){\circle*{0.5}}\put(5,39){\line(-1,-1){6}}
\put(11,33){\line(1,1){6}}\put(17,39){\circle*{0.5}}
\put(17,39){\line(1,-1){6}}\put(23,33){\circle*{0.5}}
\put(23,33){\line(-1,-1){6}}\put(17,27){\circle*{0.5}}
\put(17,27){\line(-1,1){6}}\put(4.5,24){\footnotesize{$1$}}
\put(16.5,24){\footnotesize{$8$}}\put(-2.5,30){\footnotesize{$7$}}
\put(-8,40){\footnotesize{$9$}}\put(9.5,29){\footnotesize{$10$}}
\put(21.5,29){\footnotesize{$12$}}\put(4,40){\footnotesize{$11$}}
\put(16,40){\footnotesize{$13$}}
\end{picture}
\end{center}
which is a sequence of standard Young tableaux having shape $\sigma_2,\tau_3^{d},\sigma_{8}^{d}$. The shapes $\tau_3^d$ and $\sigma_8^d$ are the dual Hasse diagrams of $\tau_3$ and $\sigma_8$, respectively. From the above sequence of tableaux, we can retrieve the pair $(\bar{\sigma}_1,(3,6,1,8))$ by applying the inverse algorithm $1$. More precisely, we first compare $2$ and $4$ contained in the first two tableaux from above, since $2<4$, we add the edge $\{2,5\}$ and remove the edge $\{2,3\}$. This gives us
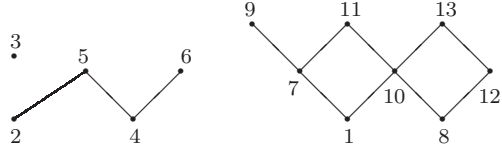
\begin{center}
\setlength{\unitlength}{3pt}
\begin{picture}(40,18)(-25,24)
\put(-37,27){\circle*{0.5}}
\qbezier(-37,27)(-32.5,30)(-28,33)
\put(-37,35){\circle*{0.5}}
\put(-37.5,24){\footnotesize{$2$}}\put(-37.5,36){\footnotesize{$3$}}
\put(-22,27){\circle*{0.5}}
\put(-22,27){\line(-1,1){6}}\put(-22,27){\line(1,1){6}}
\put(-28,33){\circle*{0.5}}\put(-22.5,24){\footnotesize{$4$}}
\put(-16,33){\circle*{0.5}}\put(-29,34){\footnotesize{$5$}}
\put(-16,34){\footnotesize{$6$}}
\put(5,27){\circle*{0.5}}\put(5,27){\line(-1,1){6}}
\put(-1,33){\circle*{0.5}}\put(-1,33){\line(-1,1){6}}
\put(-7,39){\circle*{0.5}}\put(5,27){\line(1,1){6}}
\put(11,33){\circle*{0.5}}\put(11,33){\line(-1,1){6}}
\put(5,39){\circle*{0.5}}\put(5,39){\line(-1,-1){6}}
\put(11,33){\line(1,1){6}}\put(17,39){\circle*{0.5}}
\put(17,39){\line(1,-1){6}}\put(23,33){\circle*{0.5}}
\put(23,33){\line(-1,-1){6}}\put(17,27){\circle*{0.5}}
\put(17,27){\line(-1,1){6}}\put(4.5,24){\footnotesize{$1$}}
\put(16.5,24){\footnotesize{$8$}}\put(-2.5,30){\footnotesize{$7$}}
\put(-8,40){\footnotesize{$9$}}\put(9.5,29){\footnotesize{$10$}}
\put(21.5,29){\footnotesize{$12$}}\put(4,40){\footnotesize{$11$}}
\put(16,40){\footnotesize{$13$}}
\end{picture}
\end{center}
Secondly we compare $4$ and $7$ from the above diagram. Since $4<7$, we add the edge $\{4,9\}$ and remove the edge $\{4,6\}$. This gives us
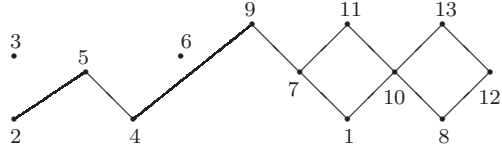
\begin{center}
\setlength{\unitlength}{3pt}
\begin{picture}(40,18)(-25,24)
\put(-37,27){\circle*{0.5}}
\qbezier(-37,27)(-32.5,30)(-28,33)
\put(-37,35){\circle*{0.5}}
\put(-37.5,24){\footnotesize{$2$}}\put(-37.5,36){\footnotesize{$3$}}
\put(-22,27){\circle*{0.5}}
\put(-22,27){\line(-1,1){6}}
\qbezier(-22,27)(-14.5,33)(-7,39)
\put(-28,33){\circle*{0.5}}\put(-22.5,24){\footnotesize{$4$}}
\put(-16,35){\circle*{0.5}}\put(-29,34){\footnotesize{$5$}}
\put(-16,36){\footnotesize{$6$}}
\put(5,27){\circle*{0.5}}\put(5,27){\line(-1,1){6}}
\put(-1,33){\circle*{0.5}}\put(-1,33){\line(-1,1){6}}
\put(-7,39){\circle*{0.5}}\put(5,27){\line(1,1){6}}
\put(11,33){\circle*{0.5}}\put(11,33){\line(-1,1){6}}
\put(5,39){\circle*{0.5}}\put(5,39){\line(-1,-1){6}}
\put(11,33){\line(1,1){6}}\put(17,39){\circle*{0.5}}
\put(17,39){\line(1,-1){6}}\put(23,33){\circle*{0.5}}
\put(23,33){\line(-1,-1){6}}\put(17,27){\circle*{0.5}}
\put(17,27){\line(-1,1){6}}\put(4.5,24){\footnotesize{$1$}}
\put(16.5,24){\footnotesize{$8$}}\put(-2.5,30){\footnotesize{$7$}}
\put(-8,40){\footnotesize{$9$}}\put(9.5,29){\footnotesize{$10$}}
\put(21.5,29){\footnotesize{$12$}}\put(4,40){\footnotesize{$11$}}
\put(16,40){\footnotesize{$13$}}
\end{picture}
\end{center}
Thirdly we remove the edges connecting to $1$ and $8$. This gives us the pair
$(\bar{\sigma}_1,(3,6,1,8))$, i.e.,
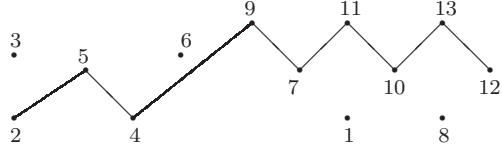
\begin{center}
\setlength{\unitlength}{3pt}
\begin{picture}(40,18)(-25,24)
\put(-37,27){\circle*{0.5}}
\qbezier(-37,27)(-32.5,30)(-28,33)
\put(-37,35){\circle*{0.5}}
\put(-37.5,24){\footnotesize{$2$}}\put(-37.5,36){\footnotesize{$3$}}
\put(-22,27){\circle*{0.5}}
\put(-22,27){\line(-1,1){6}}
\qbezier(-22,27)(-14.5,33)(-7,39)
\put(-28,33){\circle*{0.5}}\put(-22.5,24){\footnotesize{$4$}}
\put(-16,35){\circle*{0.5}}\put(-29,34){\footnotesize{$5$}}
\put(-16,36){\footnotesize{$6$}}
\put(5,27){\circle*{0.5}}
\put(-1,33){\circle*{0.5}}\put(-1,33){\line(-1,1){6}}
\put(-7,39){\circle*{0.5}}
\put(11,33){\circle*{0.5}}\put(11,33){\line(-1,1){6}}
\put(5,39){\circle*{0.5}}\put(5,39){\line(-1,-1){6}}
\put(11,33){\line(1,1){6}}\put(17,39){\circle*{0.5}}
\put(17,39){\line(1,-1){6}}\put(23,33){\circle*{0.5}}
\put(17,27){\circle*{0.5}}
\put(4.5,24){\footnotesize{$1$}}
\put(16.5,24){\footnotesize{$8$}}\put(-2.5,30){\footnotesize{$7$}}
\put(-8,40){\footnotesize{$9$}}\put(9.5,30){\footnotesize{$10$}}
\put(21.5,30){\footnotesize{$12$}}\put(4,40){\footnotesize{$11$}}
\put(16,40){\footnotesize{$13$}}
\end{picture}
\end{center}
For $n\ge 2$, let $\CMcal{T}_{3n-2}$ be the image set of pairs $(\sigma,(x_1,\ldots,x_{n-1}))\in\CMcal{A}_n$ under Algorithm $1$. For $n=1$ we assume $\CMcal{T}_{1}=\CMcal{A}_1=\{\cdot _1\}$. Since Algorithm $1$ is a bijection between $\CMcal{A}_n$ and $\CMcal{T}_{3n-2}$ for $n\ge 2$, it follows from Algorithm $1$ that $\vert \CMcal{T}_{3n-2}\vert=\vert\CMcal{A}_n\vert=(2n)(2n+1)\cdots(3n-2)E_{2n-1}$ for $n\ge 2$, and moreover $\vert\CMcal{T}_{1}\vert=\vert\CMcal{A}_1\vert=E_1$. It remains to establish
\begin{eqnarray}\label{E:bij2}
\vert \CMcal{T}_{3n-2}\vert=2^{2n-2}f^{\sigma_{3n-2}}.
\end{eqnarray}
Since every element in $\CMcal{T}_{3n-2}$ is a sequence of $3$-strip tableaux, it is much simpler to count $\CMcal{T}_{3n-2}$ via the symbolic method for the labeled cases, see Chapter $2$ of \cite{FS}. We will adopt the notations $\star,\textsc{Seq},\textsc{Seq}_k$ for the labeled cases from \cite{FS}. Their definitions are
\begin{enumerate}
\item $\CMcal{A}\star \CMcal{B}$ is the labeled product of $\CMcal{A}$ and $\CMcal{B}$, which is
obtained by forming ordered pairs from $A\times B$ and performing all possible order-consistent relabellings.
\item The $k$-th labelled power of $\CMcal{B}$ is defined as $(\CMcal{B}\star\CMcal{B}\cdots \star \CMcal{B})$, with $k$ factors equal to $\CMcal{B}$. It is denoted $\textsc{Seq}_k(\CMcal{B})$ as it corresponds to forming $k$-sequences and performing all consistent relabellings. The labelled sequence class of $\CMcal{B}$ is denoted by $\textsc{Seq}(\CMcal{B})$ and is defined by $\textsc{Seq}(\CMcal{B})=\{\epsilon\}+\CMcal{B}+(\CMcal{B}\star \CMcal{B})+(\CMcal{B}\star \CMcal{B}\star \CMcal{B})+\cdots=\bigcup_{k\ge 0}\textsc{Seq}_k(\CMcal{B})$ where $\epsilon$ represents the empty structure.
\end{enumerate}
Now we are in position to apply the symbolic method to count $\CMcal{T}_{3n-2}$. Let $\CMcal{G}$ (resp. $\CMcal{G}^{\gamma,d}$, $\CMcal{G}^{d}$) be the class of standard Young tableaux of shape $\sigma_{3n-1}$ (resp. $\sigma_{3n-1}^{\gamma,d}$, $\sigma_{3n-1}^d$) for all $n\ge 1$, let $\CMcal{F}$ be the class of standard Young tableaux of shape $\sigma_{3n-2}$ for all $n\ge 1$, let
$\CMcal{H}$ (resp. $\CMcal{H}^d$) be the class of standard Young tableaux of shape $\tau_{3n}$ (resp. $\tau_{3n}^d$) for all $n\ge 1$. Let furthermore $\CMcal{T}=\bigcup_{n\ge 1}\CMcal{T}_{3n-2}$ be the class of objects in $\CMcal{T}_{3n-2}$ for all $n\ge 1$. We observe, for every $(T_1,T_2,\ldots,T_m)\in\CMcal{T}_{3n-2}$ where $T_i$ is a standard Young tableau for every $i$, if $m=1$ and $n\ge 2$, i.e., in the Algorithm $1$, $x_i<r_i$ is true for every $i$, therefore $T_1$ must be a $3$-strip tableau of shape $\sigma_{3n-2}$ for $n\ge 2$. If $m=1$ and $n=1$, then $T_1$ is a labeled point, which is a $3$-strip tableau of shape $\sigma_1$. In both cases, $T_1$ has shape $\sigma_{3n-2}$ for $n\ge 1$, which belongs to the class $\CMcal{F}$. If $m\ge 2$, then $T_1$ and $T_m$ have shape $\sigma_{3i-1}^{\gamma,d}$ and $\sigma_{3k-1}^d$ respectively. For every $i\ne 1,i\ne m$, $T_i$ has shape $\tau_{3j_i}^d$ for some $j_i\ge 1$, namely $(T_1,T_2,\ldots,T_m)$ belongs to the class $\CMcal{G}^{\gamma,d}\star \textsc{Seq}(\CMcal{H}^d)\star \CMcal{G}^{d}$. In sum, we have
\begin{eqnarray}\label{E:sym}
\CMcal{T}=\CMcal{F}+\CMcal{G}^{\gamma,d}\star \textsc{Seq}(\CMcal{H}^d)\star \CMcal{G}^{d}.
\end{eqnarray}
Let $F(x),G(x),H(x),T(x)$ be the exponential generating function for the number $f^{\sigma_{3n-2}}$, $f^{\sigma_{3n-1}}$, $f^{\tau_{3n}}$ and $\vert\CMcal{T}_{3n-2}\vert$ respectively, i.e.,
\begin{align*}
F(x)&=\sum_{n\ge 1}\frac{f^{\sigma_{3n-2}}x^{3n-2}}{(3n-2)!},\,
G(x)=\sum_{n\ge 1}\frac{f^{\sigma_{3n-1}}x^{3n-1}}{(3n-1)!},\\
H(x)&=\sum_{n\ge 1}\frac{f^{\tau_{3n}}x^{3n}}{(3n)!},\,\quad\quad
T(x)=\sum_{n\ge 1}\frac{\vert\CMcal{T}_{3n-2}\vert x^{3n-2}}{(3n-2)!}.
\end{align*}
Here note that $G(x)$ is also the exponential generating function for $f^{\sigma_{3n-1}^{\gamma,d}}$ and
$f^{\sigma_{3n-1}^{d}}$. $H(x)$ is also the exponential generating function for $f^{\tau_{3n}^{d}}$, thus eq.~(\ref{E:sym}), in terms of the exponential generating functions, is
\begin{eqnarray}\label{E:de1}
T(x)=F(x)+\frac{G^2(x)}{1-H(x)}.
\end{eqnarray}
Moreover, from Lemma~\ref{L:funlem} we have $G(x)=xF(x)/2$. Similarly we can insert a sequence of integers to a down-up permutation. For $n\ge 2$, let $\CMcal{B}_n$ be the set of pairs $(\sigma,(x_1,\ldots,x_{n-1}))$ where $(x_1,\ldots,x_{n-1})$ is a sequence of $(n-1)$ distinct positive integers from $[3n]$, and $\sigma=\sigma_1\cdots \sigma_{2n+1}$ is a down-up permutation of the set $[3n]-\{x_1,\ldots,x_{n-1}\}$, it is clear that $\vert\CMcal{B}_n\vert=(2n+2)(2n+3)\cdots (3n)E_{2n+1}$. We will add the integers $x_1,\ldots,x_{n-1}$ successively into $\sigma$ as follows:
\begin{tabbing}
{\bf Algorithm $2$}: {Insert the sequence $(x_1,\ldots,x_{n-1})$ into the down-up permutation $\sigma$}.\\
\rule{1cm}{0mm}for $i:=1\rightarrow n-1$ do\\
\rule{2cm}{0mm}\=$r_i:=\min\{\sigma_{2i},\sigma_{2i+2}\}$;\\
\>if $x_i<r_i$ then\\
\>\rule{1cm}{0mm}\=let both $\sigma_{2i}$ and $\sigma_{2i+2}$ cover $x_i$;\\
\>else \\
\>\rule{1cm}{0mm}\=let $x_i$ cover $r_i$ and remove the edge $\{\sigma_{2i+1},r_i\}$;\\
\>end if\\
\rule{1cm}{0mm}end do
\end{tabbing}
For example, consider the pair $(\bar{\sigma}_2,(3,4,10))$ where $n=4$, $\bar{\sigma}_2=216587\,11\,9\,12$ is a down-up permutation and $(3,4,10)$ is a sequence of integers, i.e., $\bar{\sigma}_2$ is
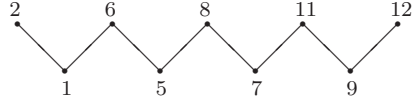
\begin{center}
\setlength{\unitlength}{3pt}
\begin{picture}(40,15)(-25,25)
\put(-32,30){\line(1,1){6}}\put(-32,30){\circle*{0.5}}
\put(-26,36){\circle*{0.5}}\put(-32,30){\line(-1,1){6}}
\put(-38,36){\circle*{0.5}}\put(-26,36){\line(1,-1){6}}
\put(-20,30){\circle*{0.5}}\put(-20,30){\line(1,1){6}}
\put(-14,36){\circle*{0.5}}\put(-14,36){\line(1,-1){6}}
\put(-8,30){\circle*{0.5}}\put(-8,30){\line(1,1){6}}
\put(-2,36){\circle*{0.5}}\put(-2,36){\line(1,-1){6}}
\put(4,30){\circle*{0.5}}\put(4,30){\line(1,1){6}}
\put(10,36){\circle*{0.5}}
\put(-32.5,27){\footnotesize{$1$}}\put(-39,37){\footnotesize{$2$}}
\put(-27,37){\footnotesize{$6$}}\put(-20.5,27){\footnotesize{$5$}}
\put(-15,37){\footnotesize{$8$}}\put(-8.5,27){\footnotesize{$7$}}
\put(-3,37){\footnotesize{$11$}}\put(3.5,27){\footnotesize{$9$}}
\put(9,37){\footnotesize{$12$}}
\end{picture}
\end{center}
After applying Algorithm $2$ on the pair $(\bar{\sigma}_2,(3,4,10))$, we obtain
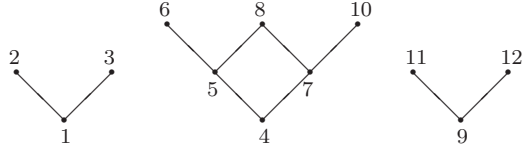
\begin{center}
\setlength{\unitlength}{3pt}
\begin{picture}(40,18)(-25,24)
\put(-37,27){\line(1,1){6}}\put(-37,27){\circle*{0.5}}
\put(-31,33){\circle*{0.5}}
\put(-37.5,24){\footnotesize{$1$}}\put(-37,27){\line(-1,1){6}}
\put(-43,33){\circle*{0.5}}\put(-32,34){\footnotesize{$3$}}
\put(-44,34){\footnotesize{$2$}}
\put(-12,27){\circle*{0.5}}\put(-12,27){\line(-1,1){6}}
\put(-18,33){\circle*{0.5}}\put(-18,33){\line(-1,1){6}}
\put(-24,39){\circle*{0.5}}\put(-18,33){\line(1,1){6}}
\put(-12,39){\circle*{0.5}}\put(-12,39){\line(1,-1){6}}
\put(-6,33){\circle*{0.5}}\put(-6,33){\line(-1,-1){6}}
\put(-6,33){\line(1,1){6}}\put(0,39){\circle*{0.5}}
\put(-12.5,24){\footnotesize{$4$}}\put(-19,30){\footnotesize{$5$}}
\put(-25,40){\footnotesize{$6$}}\put(-13,40){\footnotesize{$8$}}
\put(-1,40){\footnotesize{$10$}}\put(-7,30){\footnotesize{$7$}}
\put(13,27){\circle*{0.5}}\put(13,27){\line(1,1){6}}\put(19,33){\circle*{0.5}}
\put(13,27){\line(-1,1){6}}\put(7,33){\circle*{0.5}}
\put(12.5,24){\footnotesize{$9$}}\put(18,34){\footnotesize{$12$}}
\put(6,34){\footnotesize{$11$}}
\end{picture}
\end{center}
which is a sequence of standard Young tableaux having shape $\tau_3^d$ and $\tau_6^d$. Similarly from the above sequence of tableaux, we can retrieve the pair $(\bar{\sigma}_2,(3,4,10))$ by applying the inverse algorithm $2$. For $n\ge 2$, let $\CMcal{T}_{3n}$ be the image set of pairs $(\sigma,(x_1,\ldots,x_{n-1}))\in\CMcal{B}_n$ under Algorithm $2$. Clearly the Algorithm $2$ is a bijection between $\CMcal{B}_n$ and $\CMcal{T}_{3n}$, therefore $\vert \CMcal{T}_{3n}\vert=\vert\CMcal{B}_n\vert=(2n+2)(2n+3)\cdots(3n)E_{2n+1}$ for $n\ge 2$. For $n=0,1$, we assume $\CMcal{T}_0=\{\epsilon\}$ where $\epsilon$ is the empty structure and $\CMcal{T}_{3}$ is the set of tableaux of shape $\tau_3^d$ on the set $[3]$. Let $\CMcal{T}^1=\bigcup_{n\ge 0}\CMcal{T}_{3n}$ be the class of objects in $\CMcal{T}_{3n}$ for all $n\ge 0$, then the class $\CMcal{T}^1$ contains the missing tableaux of shape $\tau_3^d$ and the empty structure $\epsilon$ under the Algorithm $2$. Recall that $\vert \CMcal{T}_{3n-2}\vert=(2n)(2n+1)\cdots(3n-2)E_{2n-1}$, $\vert \CMcal{T}_{3n}\vert=(2n+2)(2n+3)\cdots(3n)E_{2n+1}$ for $n\ge 2$,
$\vert\CMcal{T}_{3}\vert=f^{\tau_3}=E_3$ and $\vert\CMcal{T}_0\vert=1$, it is easy to check that $\vert\CMcal{T}_{3n-3}\vert(3n-2)=\vert\CMcal{T}_{3n-2}\vert$ holds for $n\ge 1$, i.e., $\{\cdot\}\star\CMcal{T}^1=\CMcal{T}$.
Let $T_1(x)$ be the exponential generating function for the numbers $\vert \CMcal{T}_{3n}\vert$,
\begin{eqnarray*}
T_1(x)=1+\frac{f^{\tau_3}x^3}{3!}+\sum_{n\ge 2}\frac{\vert \CMcal{T}_{3n}\vert x^{3n}}{(3n)!}
=\sum_{n\ge 0}\frac{\vert \CMcal{T}_{3n}\vert x^{3n}}{(3n)!}.
\end{eqnarray*}
The relation $\{\cdot\}\star\CMcal{T}^1=\CMcal{T}$ is equivalent to $xT_1(x)=T(x)$. From Algorithm $2$, we observe, for every $T_1\in \CMcal{T}_{0}\bigcup\CMcal{T}_{3}$ and for every $(T_1,T_2,\ldots,T_m)\in\CMcal{T}_{3n}$ where $n\ge 2$, if $m=1$, $T_1$ is either a $3$-strip tableau of shape $\tau_{3n}^d$ where $n\ge 1$ or $T_1=\epsilon$, namely $T_1$ belongs to the class $\CMcal{H}^d$ or $T_1$ is an empty structure. If $m\ge 2$, then for every $i\in [1,m]$, $T_i$ has shape $\tau_{3j_i}^d$ for some $j_i\ge 1$, namely $T_i$ belongs to the class $\CMcal{H}^d$. In sum, we have
\begin{eqnarray}
\CMcal{T}^1=\{\epsilon\}+\CMcal{H}^d+\bigcup_{k\ge 2}\textsc{Seq}_{k}(\CMcal{H}^d)=\textsc{Seq}(\CMcal{H}^d)\,\mbox{ and }\, \{\cdot\}\star\CMcal{T}^1=\CMcal{T}.
\end{eqnarray}
In terms of the exponential generating functions, we get
\begin{eqnarray}\label{E:de2}
\frac{T(x)}{x}=T_1(x)=\frac{1}{1-H(x)}.
\end{eqnarray}
In combination of eq.~(\ref{E:de1}), eq.~(\ref{E:de2}) and $G(x)=xF(x)/2$, it follows that
\begin{eqnarray}\label{E:de3}
T(x)=\frac{4F(x)}{4-xF^2(x)}.
\end{eqnarray}
According to the Algorithm $1$, we have
\begin{eqnarray}\label{E:epT}
\quad\sqrt{x}T(x)=\sum_{n\ge 1}\frac{\vert\CMcal{T}_{3n-2}\vert x^{3n-3/2}}{(3n-2)!}
=\sum_{n\ge 1}\frac{\vert\CMcal{A}_{n}\vert x^{3n-3/2}}{(3n-2)!}=\sum_{n\ge 1}\frac{E_{2n-1}x^{3n-3/2}}{(2n-1)!}=\tan(x^{3/2}),
\end{eqnarray}
therefore we could solve $F(x)$ from eq.~(\ref{E:de3}), which gives us $F(x)=2/\sqrt{x}\tan(x^{3/2}/2)$
and consequently in view of eq.~(\ref{E:epT}), we get $F(2^{2/3}x)=2^{2/3}T(x)$, from which eq.~(\ref{E:bij2}) is proved and thus the proof of eq.~(\ref{E:bij1}) is complete. Next we shall prove eq.~(\ref{E:taubij}). For $n\ge 2$, we set $\vert\CMcal{M}_{3n}\vert=(2n)(2n+1)\cdots (3n)E_{2n-1}$, which counts the ways to insert two distinct integers $x_{n},x_{n+1}$ from $[3n]$ into a pair $(\sigma,(x_1,\ldots,x_{n-1}))$ where $(x_1,\ldots,x_{n-1})$ is a sequence of $(n-1)$ distinct positive integers from $[3n]-\{x_n,x_{n+1}\}$ and $\sigma=\sigma_1\sigma_2\cdots \sigma_{2n-1}$ is an up-down permutation on the set $[3n]-\{x_1,\ldots,x_{n+1}\}$. For $n=1$, we set $\vert\CMcal{M}_3\vert=3\cdot 2\cdot E_1$, which counts the ways to insert two distinct integers $x,y$ from $[3]$ to a single point labeled with $z$ where $z\in[3]-\{x,y\}$. This implies $\vert\CMcal{M}_{3n}\vert=(3n)(3n-1)\vert\CMcal{T}_{3n-2}\vert$ for $n\ge 1$ and therefore the exponential generating functions for the numbers $\vert\CMcal{M}_{3n}\vert$ is $x^2T(x)$. Now it remains to prove $\vert\CMcal{M}_{3n}\vert=(2^{2n}-1)f^{\tau_{3n}}$. From eq.~(\ref{E:de2}) and eq.~(\ref{E:epT}), we get
$H(x)=1-x^{3/2}/\tan(x^{3/2})$ and consequently in view of eq.~(\ref{E:epT}) we get $H(2^{2/3}x)-H(x)=x^2T(x)$, from which it follows that $\vert\CMcal{M}_{3n}\vert=(2^{2n}-1)f^{\tau_{3n}}$ and therefore eq.~(\ref{E:taubij}) is proved.
\qed

{\bf Remark}: Algorithm $1$ enriches every up-down permutation from below. In view of the bijection between down-up permutations and up-down permutations, Algorithm $2$ actually enriches every up-down permutation from above. They are consistent with the generalization of up-down permutations via ``thickening'' the $2$-strip tableaux.

In the bijective proof of Theorem~\ref{T:3strip}, via Algorithm $1$ and $2$ we get a sequence of independent $3$-strip tableaux. It would be ideal to keep the Hasse diagram connected when inserting a sequence of integers to an up-down permutation. However, this attempt has met with frustratingly little progress. One major difficulty is to keep track of the multiplicity for every $3$-strip tableau after inserting a sequence of integers to an up-down permutation.

\end{document}